\documentclass[fleqn,10pt]{wlscirep}
\usepackage[utf8]{inputenc}
\usepackage[T1]{fontenc}

\usepackage{algorithm}
\usepackage{algorithmic}
\usepackage{subcaption}
\usepackage{changebar}
\usepackage{listings}
\usepackage{float}
\usepackage{diagbox}
\usepackage{lineno}
\usepackage{amsmath, amssymb, amsbsy, amsthm, epigraph, longtable, soul, hyperref, soul}
\usepackage{graphicx, xcolor, url, tikz}
\usetikzlibrary{fit}

\usepackage{etoolbox}

\newcommand*\linenomathpatch[1]{%
  \cspreto{#1}{\linenomath}%
  \cspreto{#1*}{\linenomath}%
  \csappto{end#1}{\endlinenomath}%
  \csappto{end#1*}{\endlinenomath}%
}

\linenomathpatch{equation}
\linenomathpatch{gather}
\linenomathpatch{multline}
\linenomathpatch{align}
\linenomathpatch{alignat}
\linenomathpatch{flalign}

\newcommand{\N}{\mathbb{N}}
\newcommand{\R}{\mathbb{R}}

\renewcommand{\phi}{\varphi}
\renewcommand{\theta}{\vartheta}

\newcommand{\calL}{\mathcal{L}}

\newcommand{\calR}{\mathcal{R}}

\renewcommand{\bar}[1]{\overline{#1}}
\newcommand{\abs}[1]{\left|#1\right|}
\newcommand{\norm}[1]{\left\|#1\right\|}

\newcommand{\lb}{\left(}
\newcommand{\lsb}{\left[}

\newcommand{\rb}{\right)}
\newcommand{\rsb}{\right]}

\newcommand{\lpm}{\begin{pmatrix}}
\newcommand{\rpm}{\end{pmatrix}}

\newtheorem{theorem}{Theorem}
\newtheorem{lemma}[theorem]{Lemma}
\newtheorem{corollary}[theorem]{Corollary}
\newtheorem{remark}{Remark}

\theoremstyle{definition}

\newtheorem{example}{Example}

\begin{document}
\title{An integro-differential model for the spread of diseases}

\author[1,*]{Moritz Schäfer}
\author[1]{Thomas Götz}
\author[2]{Karol Niedzielewski}
\author[2]{Tyll Krüger}



\affil[1]{Mathematical Institute, University of Koblenz, 56070 Koblenz, Germany}
\affil[2]{Faculty of Electronics, Department of Control Systems and Mechatronics, Wrocław University of Science and Technology, Wrocław, Poland}

\begin{abstract}
 In this study, we present an integro-differential model to simulate the local spread of infections. The model incorporates a standard susceptible-infected-recovered (\textit{SIR}-) model enhanced by an integral kernel, allowing for non-homogeneous mixing between susceptibles and infectives. We define requirements for the kernel function and derive analytical results for both the \textit{SIR}- and a reduced susceptible-infected-susceptible (\textit{SIS}-) model, especially the uniqueness of solutions.
 
 In order to optimize the balance between disease containment and the social and political costs associated with lockdown measures, we set up requirements for the implementation of control function, and show examples for three different formulations for the control: continuous and time-dependent, continuous and space- and time-dependent, and piecewise constant space- and time-dependent. Latter represent reality more closely as the control cannot be updated for every time and location. We found the optimal control values for all of those setups, which are by nature best for a continuous and space-and time dependent control, yet found reasonable results for the discrete setting as well.

To validate the numerical results of the integro-differential model, we compare them to an established agent-based model that incorporates social and other microscopical factors more accurately and thus acts as a benchmark for the validity of the integro-differential approach. A close match between the results of both models validates the integro-differential model as an efficient macroscopic proxy. Since computing an optimal control strategy for agent-based models is computationally very expensive, yet comparatively cheap for the integro-differential model, using the proxy model might have interesting implications for future research.

\end{abstract}

\maketitle

\section{Introduction}

During the last three years, COVID-19 has shown that there is an increasing value in accurate models for the local and global spread of diseases, also giving advice to policy makers on how to deal with them, cf. e.g. Bracher et al. \cite{Bra21,Bra212}, Sherratt et al. \cite{She23} and Priesemann et al. \cite{Pri21}.

Country and region based statistics from many countries, e.g. Germany and Poland, show that regionally contained cases can spread throughout the country in a short period of time, especially during the initial phase(s) of the disease in spring and summer 2020. Lockdowns and other social restrictions were imposed as a result in order to contain the infection numbers, to relieve the strain on the health system, and to reduce the amount of severely ill or dead. How the spread of infections can be explained and how well measures actually work is still an open or widely discussed question in many countries. 

The local or regional spread of infections has been adressed in many previous works. Kuehn and Mölter\cite{Kue22} investigate transport effects on epidemics using two coupled models, a static epidemic network and a dynamical transport network, also with non-local, fractional transport dynamics. They find that transport processes induce additional spreading ways and that way lowers the epidemic threshold; generalising the process to fractional or non-local dynamics, however, raises the epidemic threshold. In several papers, the local spread of infections is modelled by PDE (partial differential equation) models. Viguerie et al.\cite{Vig21} argue that their geographical model simulations could be used to inform authorities to design effective measures and anticipate the allocation of important medical resources. Wang and Yamamoto\cite{Wan20} provide a forecasting model for COVID-19 using Google mobility data and PDE models, as well as find acceptable validity results of their model by comparison with COVID-19 data. A fractional PDE modelling of the spatial spread in Corona can be found in Logeshwari et al. \cite{Log20}, where they designed a system to predict the outcome of viral spreading in India. Harris and Bodman \cite{Har22} investigate the spread through a country with different regions of different densities. A diffusion-based and non-international approach can be found in Berestycki et al. \cite{Ber20}. The authors find that fast diffusion effects along major roads are an important factor of the spread of epidemics like COVID-19 in Italy and HIV in the Democratic Republic of Congo. In another upcoming paper, Schäfer and Heidrich\cite{Sch23} analyse the local spread of COVID-19 infections in a German district by another susceptible-exposed-infected-recovered (SEIR) model including PDEs. 

\paragraph{Structure of the paper.}

In this paper, we model the local spread of infections by an integro-differential model. Instead of the typical homogeneous mixing between susceptibles and infectives, a classical susceptible-infected-recovered (\textit{SIR}) model is enhanced by an integral kernel. The kernel function should depend on a space-. We present a proof for the uniqueness of solutions for the model. Lockdowns and other measures are included in our model(s) by a control function, which can be optimised under various assumptions: On the one side, it is aimed to contain the disease as much as possible, on the other side, we also consider the social and political costs of a lockdown, especially when case numbers are (comparatively) low, while attention has to be paid to not overload the health capacities and other problems like Long-COVID or economic problems given large infection numbers. We make use of three different control functions: a time-dependent, a continuous space- and time-dependent, and a piecewise constant space- and time-dependent control. In the following, we define the required target function for the optimisation of the "lockdown" control and present the corresponding Forward-Backward method. In order to validate the numerical results of the integro-differential model, we compare them to those of an established agent–based model in which social factors can be implemented more accurately. The macroscopical outcome of our model is compared to those of the microscopic agent-based model, which we interprete as a kind of 'ground truth'. If the results of both match well enough, we can see our integro-differential model as a macroscopic proxy model for the computatively expensive agent-based model.


\section{Integro-differential SIR model}
\subsection{Model formulation}

The basis of our model is the
SIR model introduced by Kermack and McKendrick~\cite{KMK91} consisting of the compartments $S$, $I$, $R$, which have the following meanings:
\begin{itemize}
    \item Susceptibles $S$: Depending on the transmission route, these individuals can become infected with the infectious disease when contact occurs. 
    \item Infected $I$: These individuals are infected with the disease and infectious. Contact with a susceptible individual can therefore lead to transmission of the disease.
    \item Recovered $R$: After surviving an infection, individuals are considered recovered. These individuals can no longer transmit the disease or get infected. 
\end{itemize}
The total number of individuals $N=S+I+R$ is assumed to be constant. We normalize the three compartments $S$, $I$, and $R$ by dividing all rows by $N$, resulting in $s:=\frac{S}{N}$, $z:=\frac{I}{N}$, $r:=\frac{R}{N}$ with $s+z+r=1$ (in order to avoid confusion, we use a different lower case letter $z$ instead of $i$). Following the model by Kermack and McKendrick, we assume the pathogen is transmitted from infected persons to susceptible persons at a time--independent rate $\beta>0$ and a recovery rate of $\gamma>0$ so that loss of infectivity is gained after $\gamma^{-1}$ days. Then, replacing $s=1-z-r$, the relative $sir$-model for each time point $t\in [0,T]\subset \R$ and point in space $x\in [0,1]^n \subset \R^n$ as follows:
\begin{subequations}\label{E:SIR}
\begin{align}
\frac d {d t} z(t,x)&=\beta (1-z-r) z-\gamma\, z& z(t=0,x)\,=\,z_0(x)\\
\frac d  {d  t} r(t,x)&=\gamma \, z& r(t=0,x)\,=\,r_0(x)
\end{align}
\end{subequations}
This means that the disease dynamics in a certain point $x$ would entirely depend on the initial relations $z_0$ and $r_0$ and the parameters $\beta$ and $\gamma$. To include interaction between the spatial points, we replace the factor $z$ in the term $\beta (1-z-r) z$ by an integral kernel function $k(t,x-y)$ which depends on the time and the distance between $x$ and $y$: 
\begin{subequations}\label{eq_zr}
\begin{align}
\frac d  {d  t} z(t,x)&=\beta (1-z-r)   \int_0^1 z(t,y)\,k(t,x-y)\, dy-\gamma\, z& z(t=0,x)\,=\,z_0(x)\\
\frac d  {d  t} r(t,x)&=\gamma \, z& r(t=0,x)\,=\,r_0(x)
\end{align}
\end{subequations}
 For the purpose of reasonable modelling of scenarios, the kernel $k$ should consist of three terms as follows:
\begin{itemize}
\item an \emph{space-dependent} part $a(x-y)$ which is monotonously decreasing wrt $\abs{x-y}$, e.g., an exponential function decreasing with the distance, i.e., $a(x-y)= c\,e^{-\delta\abs{x-y}}$. This part can be controlled with 
\item a \emph{control function} $u(t)\in \mathcal U = C([0,1])$ which represents the effectiveness of non-pharmaceutical interventions (lockdown, school closings, obligation of wearing masks etc.). Here, $u(t)\equiv 0$ implies no regulations and $u(t)\equiv 1$ implies total lockdown.
\item a \emph{non-adjustable} part $k_0$ which represents the fraction of transmission or a kind of 'background noise' you cannot control, e.g. household related infections. We also assume that this fraction does not depend on the spatial distance as interactions between distances can be prevented by political or social measures. For a more detailed view on the importance of households, cf. Dönges et al.\cite{Don23}.
\end{itemize}
These considerations lead to this formula for the kernel $k$:
\begin{equation}
k(t,x-y)=(1-u(t))\cdot a(x-y)+k_0
\end{equation}
For the following, assume that the kernel is independent of time $t$, i.e. $u(t)$ is constant over time. No loss of generality is effected when considering the case, reducing $k(t,x-y)$ to $k(x-y)=a(x-y)+k_0$ for the upcoming. The following assumptions regarding the interaction kernel $k$ should be met:
\begin{enumerate}
\item $k$ is continuous.
\item $k$ is non--negative.
\item $k(0)=k_0>k>0$.
\item $k$ is monotonically decreasing wrt $\abs{x-y}$.
\item $k_1:=\norm{k}_1 = \int_0^1 k(r)\, dr>0$
\item $k_1<K=\max_{x\in [0,1]} \int_0^1 k(\abs{x-y})\, dy$
\end{enumerate}
Note, that in case of a strict monotonically decreasing kernel, we get $K=2 \int_0^{1/2} k(r)\, dr$.

\subsection{Uniqueness of solutions}

The existence and quality of equilibria of the integro-differential model is the main question in this subsection. Even for the comparatively 'simple' \textit{SIS}-model in the previous section, it is not possible to proof the uniqueness of equilibria using classical fixpoint theory (cf. App. A). However, we can find satisfying results for uniqueness even for the \textit{SIR}-model using the prevalence. Again, wlog, we will consider the time-independent kernel function, i.e., $k(x-y)$.
\begin{lemma} For the \textit{SIR}-model \eqref{eq_zr}, there exists exactly one equilibrium.
\end{lemma}
\begin{proof}
To proof those numerical findings for the integro-differential SIR-model \eqref{E:SIR}, we will compute its prevalence for $s_0=s(0,x)\approx 1$ and $r_0=r(0,x)\equiv 0$. Reconsider the equation for the susceptibles, i.e.,
\begin{align}
    \frac{ds} {dt}&= - s \int_0^1 k(t,x-y) z(t,y) dy.\\
    \intertext{By substituting of the equation $\frac{dr} {dt}=\gamma z$, we obtain}
  \frac{ds} {dt}&= - \frac s \gamma \int_0^1 k(t,x-y) \,\,\frac{dr} {dt} \,dy= - \frac s \gamma\,\, \frac d {dt} \int_0^1 k(t,x-y)\, r(y) \,dy.\\
\intertext{Integrating this over $t$, it follows}
     {\ln s_\infty}&= - \frac 1 \gamma \int_0^1 k_\infty(x-y) r_\infty(y) \,dy.
\intertext{Now let $T_k$ be the integral operator on a (generalized) ground space $(S, \mu)$, whereas $S=[0,1]$ and $\mu(y)=y$, be defined by}
(T_k f) (x)&= \int_0^1 k_\infty(x-y)\,f(y) \,dy.
\intertext{Together with the necessary condition $1=s_\infty(x)+r_\infty(x)$ for all $x\in[0,1]$, we find}
  r_\infty(x)&=1-\exp\left(-\frac 1 \gamma \int_y k_\infty(x-y)\,r_\infty(y)\right)\,dy=1-\exp(-T_k r_\infty(x)).
\end{align}
This system of nonlinear equations can e.g. be solved numerically. Uniqueness can be shown using the paper of Bollobás-Janson-Riordan\cite{Bol07}: Following their Theorem 6.1, if 
\begin{align}
\norm{T_k}:=\sup\{\norm{T_k f}_2: f\geq 0, \norm f_2\leq 1\}<1,
\end{align}
the equation only has the zero solution; if $1\leq\norm {T_k}<\infty$, and $k_\infty$ is irreducible, then the equation has a unique non-zero solution for the prevalence. 

This uniqueness result for the prevalence can be transferred to the uniquess of the solution: Assume there exist two solutions $(s_1,z_1,r_1)$ and $(s_2,z_2,r_2)$ with the same initial conditions, which have the same prevalence $r_\infty$. Consider the difference function $\tilde z (t):=(z_1-z_2)(t)$ which must satisfy $\tilde z(0)=0$. Then the solution is $z(t,x)\equiv r(t,x)\equiv 0$, and the two solutions are equal.
\end{proof}
As an addition, this provides a nice definition for basic reproductive number $\calR_0$: By plugging this ansatz using the next-generation method\cite{Wat02}, we find $\calR_0=\frac \beta \gamma \norm k_2$, so that it also depends on the kernel function $k:[0,1]\to \R$.

\section{Optimisation}
\subsection{Time-dependent control}

In this article, we restrict our research on the case $n=1$; higher-dimensional models will be introduced in future research.  In a first formulation of the optimal control problem, we aim to minimize the total amount of infectives, along keeping the costs, i.e. the control $u(t)$, as low as possible. In order to maintain convexity of the problem and avoid bang-bang controls due to linearity in $u$, the cost function term is squared. Also, case numbers should be kept under a certain threshold $z_{\max}$, otherwise the capacities of the medical infrastructure can be exceeded. This could be either modelled locally, i.e., $z(t,x)\leq z_{\max,1}$ or globally, i.e., $\int_0^1 z(t,x)\, dx \leq z_{\max,2}$. If $z_{\max,1}\leq z_{\max,2}$ for all $x$, the two terms fall together, which is also assumed here for the sake of simplicity. Defining $\mathcal U$ as the set of continuous functions $u \in \mathcal U=\mathcal C([0,T])$, we find the following minimization problem:
\begin{align}\label{eq_simple_}
\notag  \min_{u(t)\in \mathcal U}\, J(u,z)&=\min_{u(t)\in \mathcal  U}\int_0^T\int_0^1 z(t,x) \,dx \,dt +\frac \eta 2 \int_0^T u^2(t) \,dt\\
 \text{subject to} \qquad& 0\leq u(t) \leq 1,\\
\notag &z(t,x)\leq z_{\max}. 
\end{align}
For implementation of the constraint $z<z_\text{max}$, we add a sigmoidal term $\psi: [0,1]\rightarrow \R^+$ in the cost functional, which holds $\psi(z>0)\approx 0$ and $\psi(z<0)\gg 1$. Using the function $\psi\left(z-z_\text{max}\right)$, case numbers $z>z_\text{max}$ are 'punished' severely. Also, if case numbers are low, we aim to 'punish' larger values for $u(t)$ because social restrictions are less accepted by the population and the political costs to implement harder restrictions increase. We therefore also define a threshold $z_{\text{min}}$ under which political costs are assumed to be high, i.e., $\psi(z>z_\text{min})$, such that case numbers $z<z_\text{min}$ are also punished severely. This results in the following maximization problem:  
\begin{align} 
\label{eq_hard}  
\notag \min_{u(t)\in \mathcal U}J(u,z)=&\int_0^T\int_0^1 z(t,x) \,dx \,dt +  \frac \eta {2}  \int_0^T u^2(t)\int_0^1 \left[1+\frac{c_1}2\psi\left(z_\text{min}-z(t,x)\right)\right] \,dx\,dt \\
&+ {\frac \omega 2}\int_0^T\int_0^1{c_2}\psi\left(z(t,x)-z_\text{max}\right)\, dx\,dt\\
\notag \text{subject to} \qquad& 0\leq u(t) \leq 1.
\end{align}
Alternatively, the control $u$ can also depend on space, i.e., let $ \mathcal {\tilde U}$ be the set of continuous functions $u \in \mathcal U=\mathcal C([0,T]\times[0,1]\rightarrow [0,1])$. Then the minimization problem reads as follows:
\begin{align} 
\label{eq_hard1}  
\notag \min_{u(t)\in \mathcal U} J(u,z)=&\int_0^T\int_0^1 z(t,x) \,dx \,dt +  \frac \eta {2}  \int_0^T\int_0^1  u^2(t,x) \left[1+\frac{c_1}2\psi\left(z_\text{min}-z(t,x)\right)\right] \,dx\,dt \\
&+ {\frac \omega 2}\int_0^T\int_0^1{c_2}\psi\left(z(t,x)-z_\text{max}\right)\, dx\,dt\\
\notag \text{subject to} \qquad& 0\leq u(t,x) \leq 1.
\end{align}
On a discrete level, solving the above minimisation problems might be complicated. On a continuous level, we can introduce the Lagrangian function (see also Lenhart and Workman \cite{Len07} for further information):
\begin{align}
\notag \mathcal L(z,u,r)=&\int_0^T\int_0^1 \lambda_1(t,x)\left[z'(t,x)-(1-z(t,x)-r(t,x))\int_0^1z(t,y)\,k(t,x-y)\,dy + \gamma \, z(t,x)\right]\, dx\, dt\\
 +&\int_0^T\int_0^1 \lambda_2(t,x)\,\left[r'(t,x)-\gamma\,z(t,x)\right]dx\,dt\\\notag-&\,J(z,u) 
\end{align}
We now want to find the stationary points of the partial derivatives of $\calL$ with respect to $u$, $z$, and $r$:
\begin{subequations}
\begin{align}
\notag\frac{\partial \calL}{\partial u}&=  \int_0^T \left\{\int_0^1 \lambda_1(t,x)\,(1-z(t,x)-r(t,x))\int_0^1 z(t,y)\, a(x-y) \,dy \right.&\\ &\hspace{1.1cm}\left.-\eta\, u(t) \int_0^1 \left[1+\frac{c_1} 2 \psi\left(z_\text{min}-z(t,x)\right)\right]\,dx\right\}\, dt& \overset {!} =\, 0 \\
 \notag\frac{\partial \calL}{\partial z} &= \int_0^T \int_0^1 \left\{-\lambda_1'(t,x)+\lambda_1(t,x)\int_0^1 z(t,y)\, k(t,x-y)\,dy\right.&\\
 \notag&\hspace{1.579cm}-\lambda_1(t,x)\,(1-z(t,x)-r(t,x))\int_0^1 k(t,x-y)\,dy +\gamma\, \lambda_1(t,x)\\
 &\hspace{1.579cm}-\gamma\,  \lambda_2(t,x) \notag\\
&\hspace{1.579cm}-1+\frac {c_1\,\eta} {4}  u^2(t)\cdot \psi'\left(z_\text{min}-z(t,x)\right)\left.-\frac {c_2\,\omega} {2}  \psi'\left(z(t,x)-z_{\max}\right)\right\} \,dx\,dt&\overset {!} =\, 0\\
 \notag\frac{\partial \calL}{\partial r} &=
 \int_0^T\int_0^1 \left\{\lambda_1(t,x)\int_0^1 z(t,y)\,k(t,x-y) \, dy\right. \\
 &\hspace{1.579cm}\left.-  \lambda_2'(t,x)\right\}\, dx\,dt&\overset {!} =\, 0
\end{align}
\end{subequations}
For the second and third equation we swapped the integrals and performed partial integration with respect to time $t$. This leads us to the following system:
\begin{subequations}\label{eq_fb_1}
\begin{align}
 z'(t,x)&=(1-z(t,x)-r(t,x)) \int_{0}^1 z(t,y)\,k(t,x-y)\, dy-\gamma\, z(t,x)&z(t=0,x)\,&=\,z_0(x)\\
 r'(t,x)&=\gamma\,z(t,x)&r(t=0,x)\,&=\,r_0(x)\\
\notag\lambda_1'(t,x)&=\lambda_1(t,x)\left[\int_0^1 z(t,y)\, k(t,x-y) \,dy-\,(1-z(t,x)-r(t,x))\int_0^1 k(t,x-y)\,dy+\gamma\right]&&\\ \notag& -\gamma\,\lambda_2(t,x)\\
&-1+\frac {c_1 \,\eta} {4}\,  u^2(t)\cdot\psi'\left(z_\text{min}-z(t,x)\right)\left.-\frac {c_2\,\omega} {2}  \psi'\left(z(t,x)-z_\text{max}\right)\right\}&\lambda_1(T,x)&=0\\
\lambda_2'(t,x)&=\lambda_1(t,x)\int_0^1 z(t,y)\, k(t,x-y) \,dy&\lambda_2(T,x)&=0\\
u(t)&=\frac { \displaystyle\int_0^1 \lambda_1(t,x)\,(1-z(t,x)-r(t,x))\left(\displaystyle\int_0^1 z(t,y)\, a(x-y) \,dy\right) \,dx}{\eta\displaystyle\int_0^1\left[1+\displaystyle\frac{c_1}2\psi\left(z_\text{min}- z(t,x)\right)\right]\,dx}&&
\end{align}
\end{subequations}
This is the so-called Forward-Backward sweep method according to the method described in Lenhart and Workman\cite{Len07}. For convergence and stability results, also see Hackbusch \cite{Hac78}. Starting with an initial guess of the control $u$ over the entire interval, e.g., $u(t)\equiv 0.5$ or $u(t,x)\equiv 0.5$, the forward problem is solved according to the differential equations for first solution of $z$ and $r$. The transversality conditions $\lambda_1(T)=\lambda_2(T)=0$ and the values for $u$ , $z$ and $r$ are used to solve the backward problem for $\lambda_1$ and $\lambda_2$. Using the results for $\lambda_1$, $\lambda_2$, $z$, and $r$, we calculate an update $\hat u$ on the time-dependent control function. The update of $u(t)$ is done by moving only a fraction $\sigma$ of the previous $u_\text{old}$ towards $\hat u(t)$:
\begin{align}
u(t)=(1-\sigma)\, u_\text{old}(t)+\sigma\, \hat u(t) \qquad \text{for all } t \in [0,T]
\end{align}
 This procedure will be repeated until the norm of two subsequent controls is 'close enough', i.e. $\norm{u-u_\text{old}}<\text{TOL}$. By numerical experiments a choice of $\sigma=0.1$ provided decent results which were convergent and the target function is monotonously decreasing with respect to the iteration.
 
\subsection{Space- and time-dependent control}\label{sec_spt}

Assume now that $u$ is also dependent on space, i.e., the kernel function reads as
\begin{equation}
k(t,x-y,x)=(1-u(t,x))\cdot a(x-y)+k_0.
\end{equation}
For this space-dependent formulation, we replace $u(t)$ by $u(t,x)$ in the previous equations. Regarding $\frac{\partial \calL}{\partial u}$, this leads to the following formulation:
\begin{align}
\notag\frac{\partial \calL}{\partial u}&=  \int_0^T\int_0^1 \left\{\lambda_1(t,x)\,(1-z(t,x)-r(t,x))\int_0^1 z(t,y)\, a(x-y) \,dy \right.&\\ &\hspace{1.579cm}\left.-\eta\, u(t,x) \left[1+\frac{c_1} 2 \psi \left(z_\text{min}-z(t,x)\right)\right]\right\}\,dx\, dt& \overset {!} =\, 0
\end{align}
and while the formulas for $z$ and $r$ remain the same, we find
\begin{subequations}\label{eq_fb_2}
\begin{align}
\notag\lambda_1'(t,x)&=\lambda_1(t,x)\left[\int_0^1 z(t,y)\, k(t,x-y) \,dy-\,(1-z(t,x)-r(t,x))\int_0^1 k(t,x-y)\,dy+\gamma\right]&&\\ & -\gamma\,\lambda_2(t,x)\\\notag
&-1+\frac {c_1 \,\eta} {4}\,  u^2(t,x)\cdot \psi'\left(-\,z_\text{min}-z(t,x)\right)\left.-\frac {c_2\,\omega} {2}  \psi'\left( z(t,x)-z_{\max}\right)\right\}&\lambda_1(T,x)&=0\\
\lambda_2'(t,x)&=\lambda_1(t,x)\int_0^1 z(t,y)\, k(t,x-y) \,dy&\lambda_2(T,x)&=0\\
u(t,x)&=\frac { \displaystyle \lambda_1(t,x)\,(1-z(t,x)-r(t,x))\displaystyle\int_0^1 z(t,y)\, a(x-y) \,dy}{\eta\left[1+\displaystyle\frac{c_1}2\psi\left(z_{\min}-z(t,x)\right)\right]}&&
\end{align}
\end{subequations}

\subsection{Space- and time-dependent, discretized control}

While the concept of a space- and time-dependent control is certainly reasonable, it is not realistic to design the control in a continuous way. We therefore consider a control function $u(t,x)$ that is designed as a piecewise constant function in both time and space. The control function $u(t,x)$ takes different constant values over different rectangular regions. Let's denote the control value within each rectangle as $u_{ij}$, where $i$ represents the time interval number and $j$ represents the spatial interval number. 

Mathematically, we can express the piecewise constant control function as follows: Let $t_0 < t_1 <  \ldots < t_{n-1} < t_n$ be the time instants that define the intervals, and $x_0 < x_1 <  \ldots < x_{m-1} < x_m$ be the spatial locations that define the intervals. Then 
\begin{align}
u(t,x) =& u_{ij}  \qquad\text{ if } t \in [t_{i-1}, t_i) \text{ and } x \in [x_{j-1}, x_j)
\end{align}
represents the control value within the corresponding rectangular region for $i=1\dots n$ and $j=1\dots m$. For the piecewise constant functions, we use the \textit{starting value} for the time interval $[t_{i-1},t_i)$, i.e., $t_{i-1}$, and the \textit{average} of the space interval $[x_{j-1}, x_j)$. This is then plugged into the spatial model as described in section \ref{sec_spt}.

\section{Agent-based model}

In order to validate the model, we compare the above described integro-differential model to a stochastic, microscopic agent-based model developed at the  Interdisciplinary Centre for Mathematical and Computational Modelling at the University of Warsaw. Complete details of this model are given in Niedzielewski et al.\cite{Nie22,Nie23}. In this model, agents have certain states (susceptible, infected, recovered, hospitalized, deceased, etc.) and infection events occur in certain contexts, i.e., on the streets, workplaces, and several more. A similar comparison can be found in Dönges et al.\cite{Don23}, but to ignore in-household transmission we use only single household contexts. Location space is one dimensional with 100 location points that are available. The single households are distributed uniformly in space. The agents are assigned individually to households and corresponding street context (only these two types of contexts are in use). When probability of infection is computed every street context infectivity is taken into account. To allow for control of diffusion of infected throughout the whole space, we also use the transmission kernel function $k(t,x-y)$. As a result, the infectivity decreases with distance between location of agent and street context. Since the ODE--model~\eqref{E:SIR} is a variant of an SIR-model, the agent-based model also just uses the SIR-states and ignores all other states. The agent-based model uses a recovery time for each infected individual that is sampled from an exponential distribution with mean $10$ days. 

\section{Numerical Simulations}

In this section, the Lagrangian optimization of the integro-differential model as of eqns. \eqref{eq_hard} and \eqref{eq_hard1} is presented and the numeric results are shown. We denote the initial condition function as follows:
\begin{align*}
z_0^{1}(x)&\equiv 2 \cdot 10^{-5},\\
z_0^2(x)&=\begin{cases} 1 \cdot 10^{-5}&x<0.9\\ 1 \cdot 10^{-4} & x\geq 0.9 \end{cases}.
\end{align*}
For reasons of comparability, at a choice of 100 spatial grid points almost the same mass is used for the initial infected in both variants (as the average of $z_0^2$ is equal to $\bar z_0^2= 1.9 \cdot 10^{-5}$). We listed all six model simulations in Tab. \ref{tab:sim}, including parameter values for the optimal control as of system \eqref{eq_hard}. Also, we choose parameter values of $c=\delta=50$ (resulting in a kernel-based reproductive number of $\calR_0 \approx 2$), $\beta=\gamma=0.1$, $c_1=c_2=1000$, $z_{\min}=1\cdot 10^{-5}$, and $z_{\max}=5\cdot 10^{-3}$. For both $\eta$ and $\omega$, we choose two different values which are described in Tab. \ref{tab:sim}. Lastly, wrt the penalization, we use the function  $\psi(z)=1+\tanh \left(1000 \,z\right)$ with its derivative $\psi'(z)=1000 \,\text{sech}^2(z)$.

In Sims. A, no space-dependent control is included in the model and the initial values are constant. In Sims. B, again the control is only time-dependent, but we induce an 'infection wave' at one of the boundaries. In Sims. C, we also induce this infection wave at the boundary, but additionally allow a control depending on both space and time. Two different choices for the parameter values $\eta$ and $\omega$ as well as a different maximal duration $T$ are imposed on all of these simulations. In Sims. D, we use a space-time-dependent control, but use 10-days resp. 10-cells average to account for a more realistic representation of the control. The results are compared to 10 ABM runs for each scenario and also their mean.
\begin{table}[H]
    \centering   
    \caption{Listing of all simulations and different parameter values used for optimization of the integro-differential model.}
    \label{tab:sim}
    \begin{tabular}{c|c|c|c|c|c|c}
    Variant    & space-dependent $u$?& piecewise constant $u$?& $T$   & $z_0$    & $\eta$ & $\omega$ 
    \\
    \hline
    A1 & no & no&  $400$ & $z_0^{1}(x)$ &$0.02$ & $1$ 
    \\
    A2 & no & no& $800$ & $z_0^{1}(x)$ & $0.005$ & $0.2$  
    \\
    B1 & no & no&  $400$ & $z_0^{2}(x)$ &$0.02$ & $1$ 
    \\
    B2 & no &  no& $800$ & $z_0^{2}(x)$ & $0.005$ & $0.2$
    \\
    C1 & yes &  no& $400$ & $z_0^{2}(x)$ &$0.02$ & $1$ 
    \\
    C2 & yes & no&  $800$ & $z_0^{2}(x)$ & $0.005$ & $0.2$ 
    \\
    D1 & yes & yes&  $400$ & $z_0^{2}(x)$ &$0.02$ & $1$ 
    \\
    D2 & yes & yes&  $800$ & $z_0^{2}(x)$ & $0.005$ & $0.2$ 
    \end{tabular}
\end{table}
Choosing an arbitrary starting value for $u(x)$ or $u(t,x)$, we use the Forward-Backward sweep method as of eqns. \eqref{eq_fb_1} and \eqref{eq_fb_2} and evaluate the target function in each step. As an example, the convergence of the target function value is presented in Fig. \ref{sim_targ}.
	\begin{figure} [H]
				\centering
				\begin{subfigure}{.4\textwidth}
					\includegraphics[width=\linewidth]{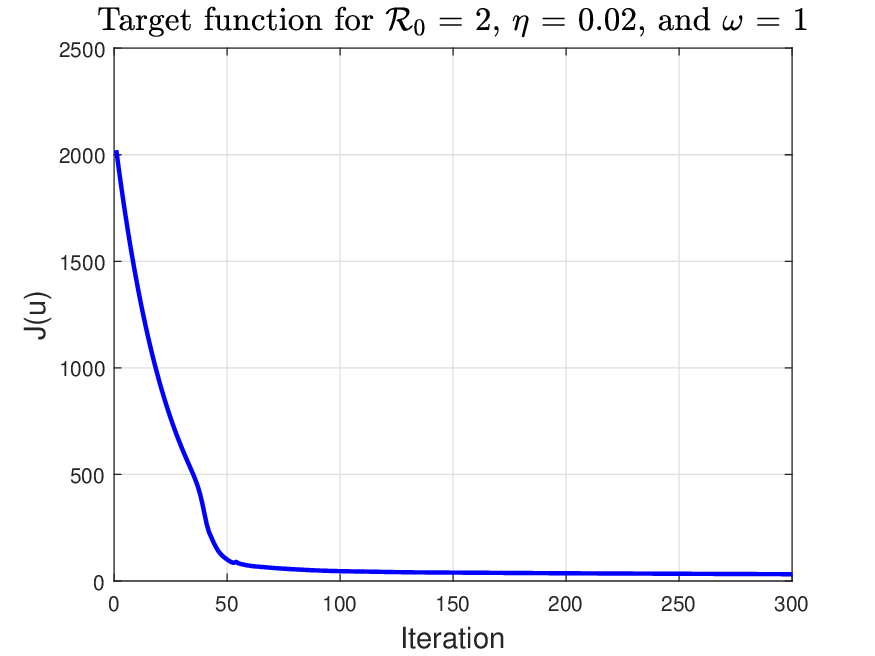}
				\end{subfigure}
				\caption{Exemplary behavior of the target function for model A1 and $u(t)\equiv 1$.}
							\label{sim_targ}
	\end{figure}
\newpage

\subsection{Simulation Results}

\subsubsection{Simulation A1}
			\begin{figure} [H]
				\centering
			\begin{subfigure}{.4\textwidth}
				\includegraphics[width=\linewidth]{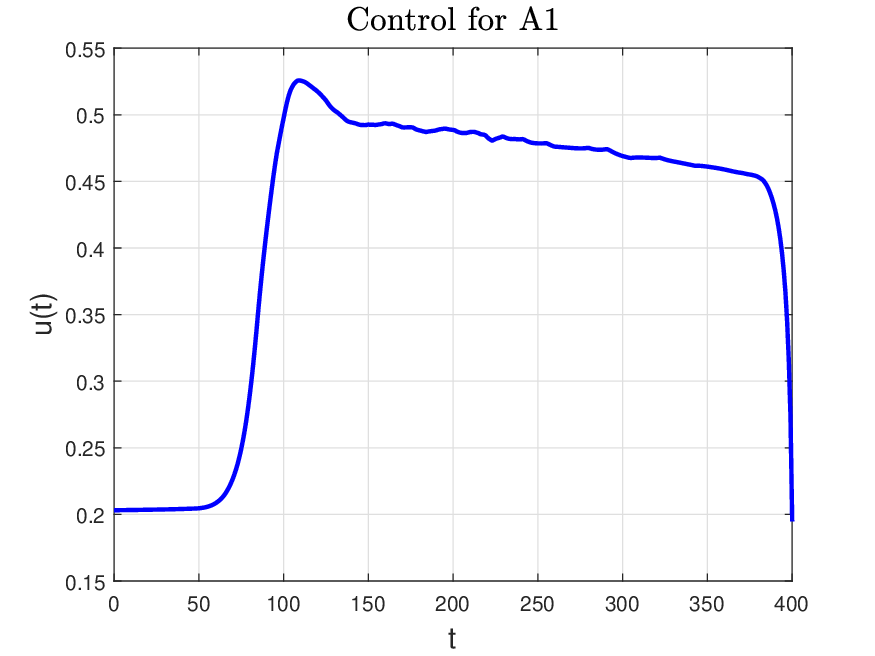}
			\end{subfigure}
				\caption{Evolution of the control in Simulation A1.}
							\label{sima13}
	\end{figure}
			\begin{figure} [H]
				\centering
			\begin{subfigure}{.4\textwidth}
				\includegraphics[width=\linewidth]{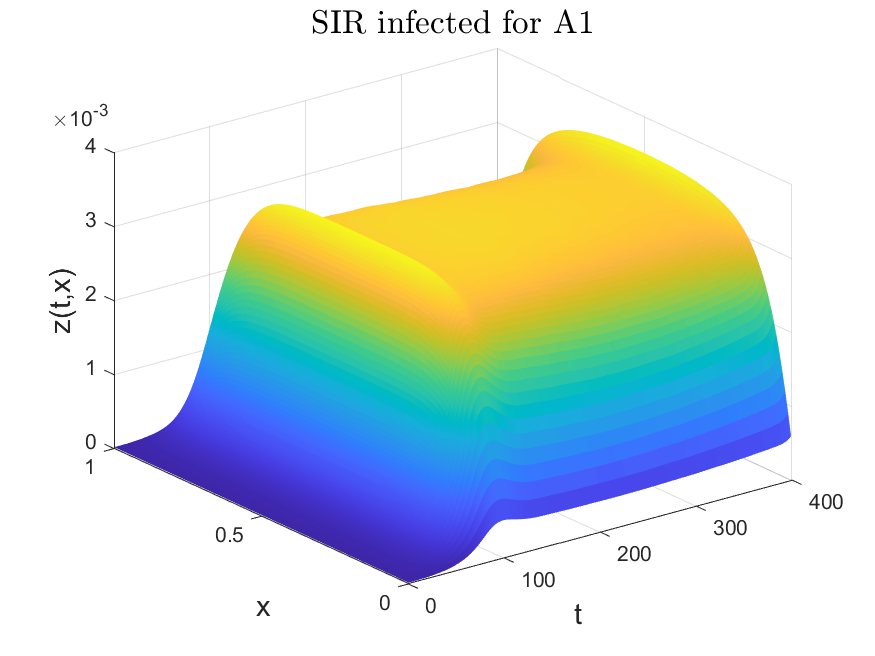}
			\end{subfigure}
				\begin{subfigure}{.4\textwidth}
					\includegraphics[width=\linewidth]{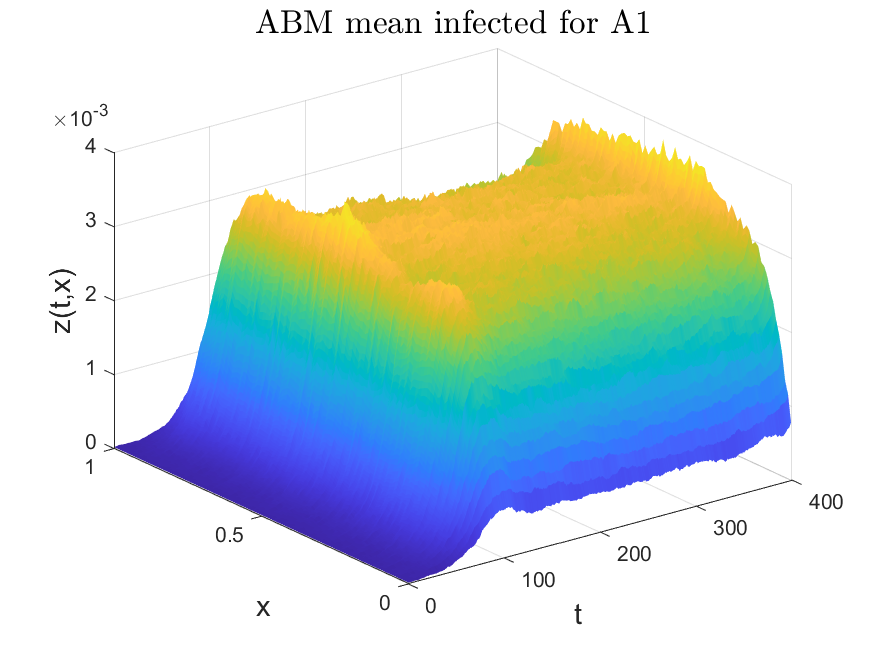}
				\end{subfigure}
				\caption{Spatio-temporal evolution of the infected in Simulation A1, on the left for the integro-differential \textit{SIR}-model, on the right for the ABM model.}
							\label{sima11}
	\end{figure}
			\begin{figure} [H]
				\centering
			\begin{subfigure}{.4\textwidth}
				\includegraphics[width=\linewidth]{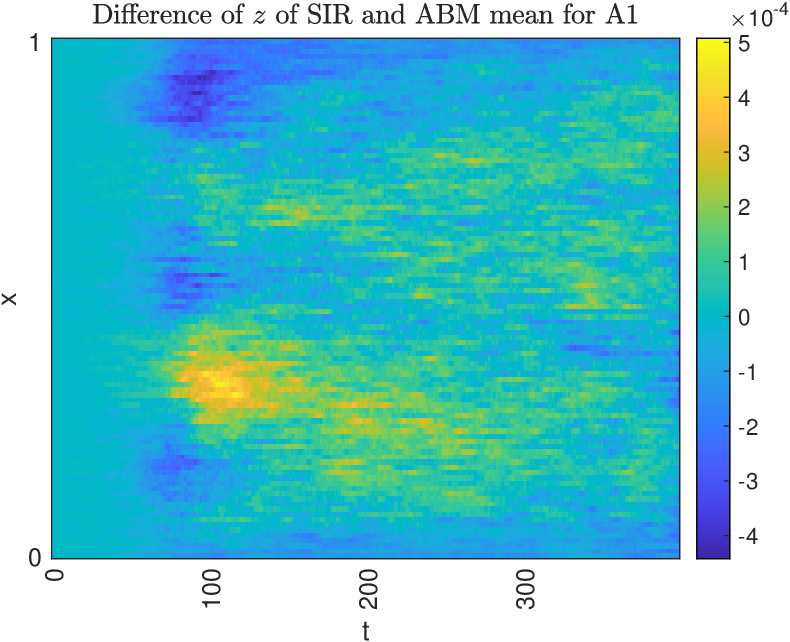}
			\end{subfigure}
				\begin{subfigure}{.4\textwidth}
					\includegraphics[width=\linewidth]{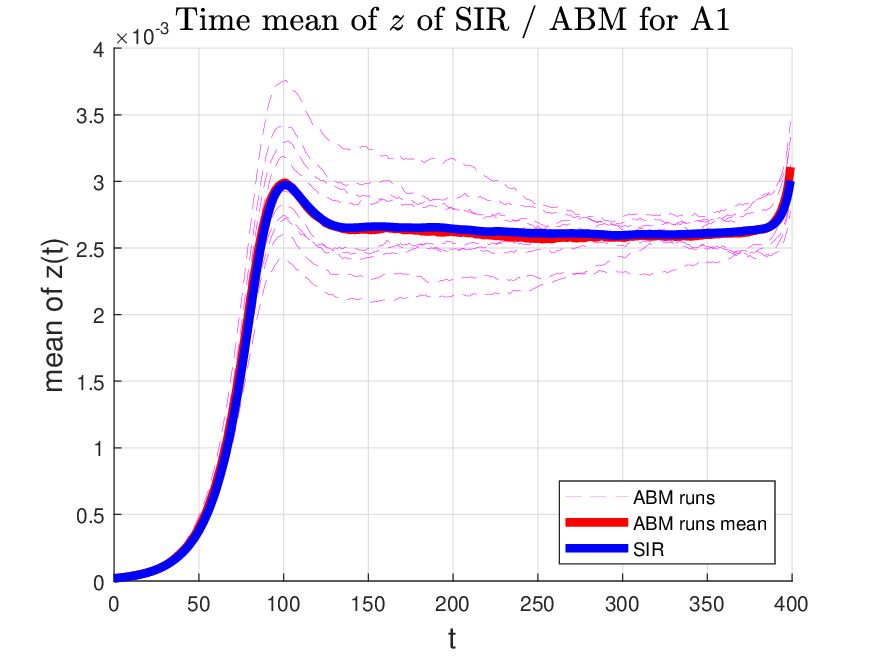}
				\end{subfigure}
				\caption{Difference between the \textit{SIR}-model to the ABM model mean (left) and temporal evolution of the spatial mean in the \textit{SIR}-model and all single runs of the ABM model, as well as their mean (right) in Simulation A1.}
							\label{sima13B}
	\end{figure}
	
	\subsubsection{Simulation A2}
	
			\begin{figure} [H]
				\centering
				\begin{subfigure}{.4\textwidth}
					\includegraphics[width=\linewidth]{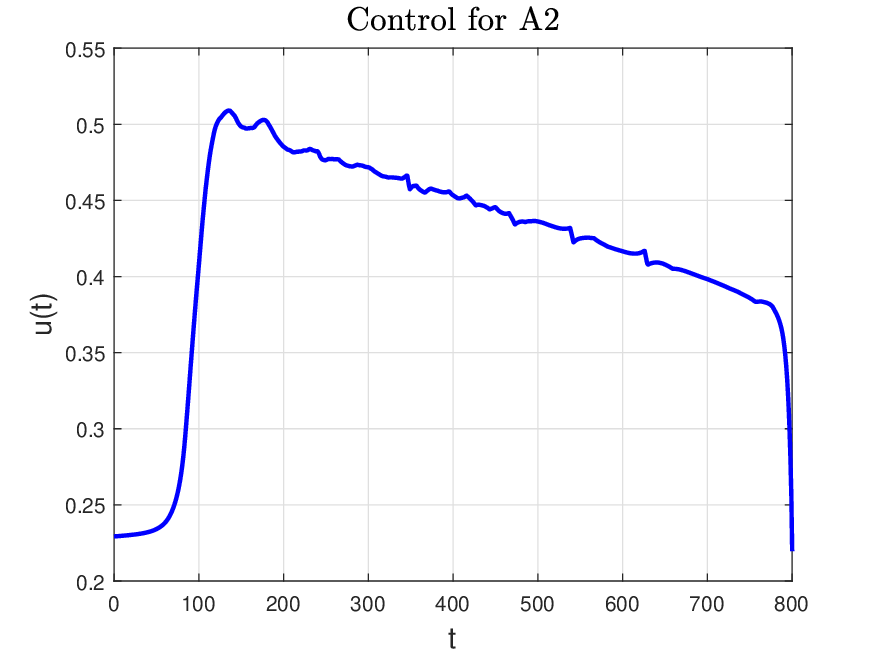}
				\end{subfigure}
				\caption{Evolution of the control in Simulation A2.}
							\label{sima23}
		\begin{figure} [H]
				\centering
			\begin{subfigure}{.4\textwidth}
				\includegraphics[width=\linewidth]{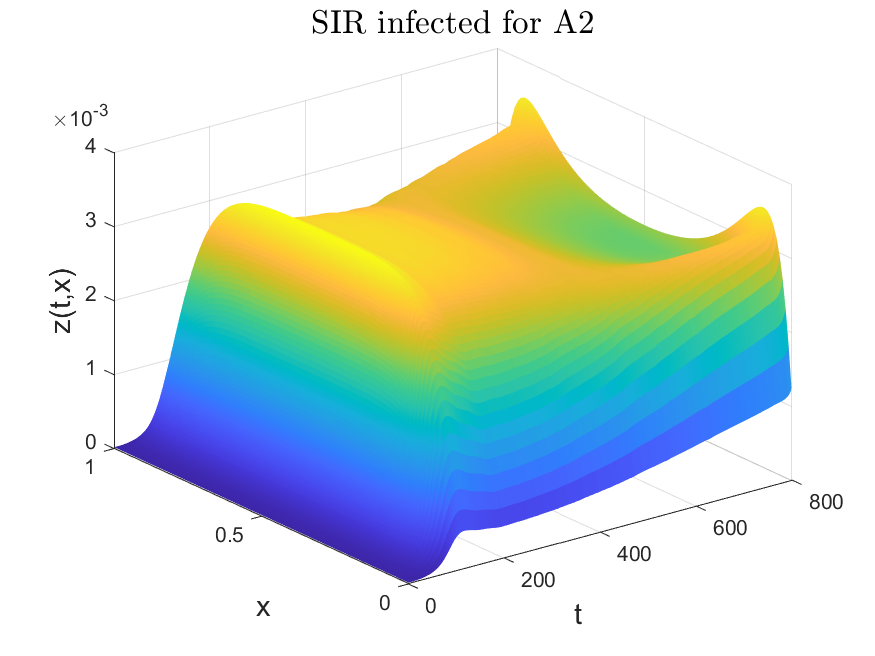}
			\end{subfigure}
				\begin{subfigure}{.4\textwidth}
					\includegraphics[width=\linewidth]{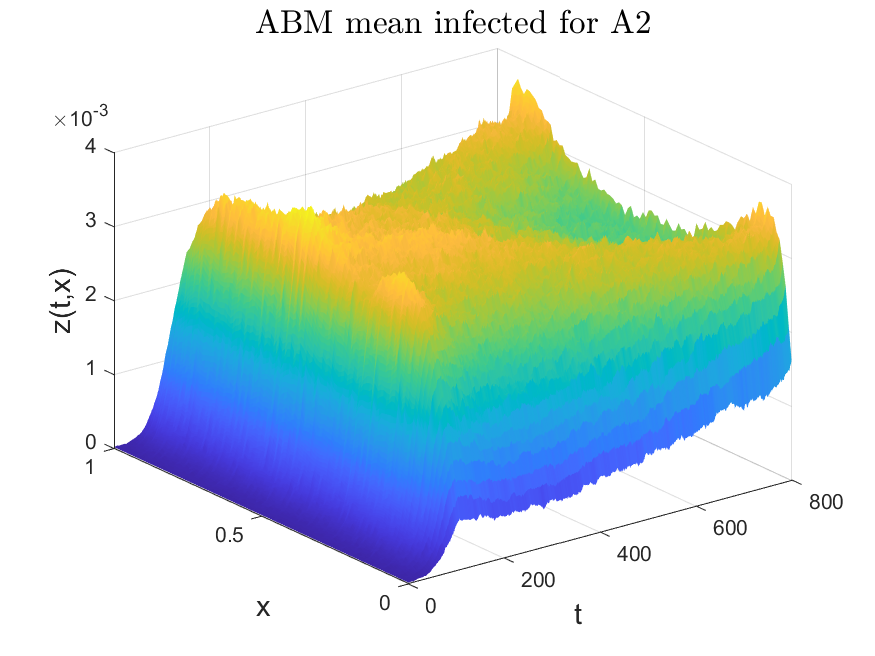}
				\end{subfigure}
				\caption{Spatio-temporal evolution of the infected in Simulation A2, on the left for the integro-differential \textit{SIR}-model, on the right for the ABM model.}
							\label{sima21}
	\end{figure}
	\end{figure}
		\begin{figure} [H]
				\centering
			\begin{subfigure}{.4\textwidth}
				\includegraphics[width=\linewidth]{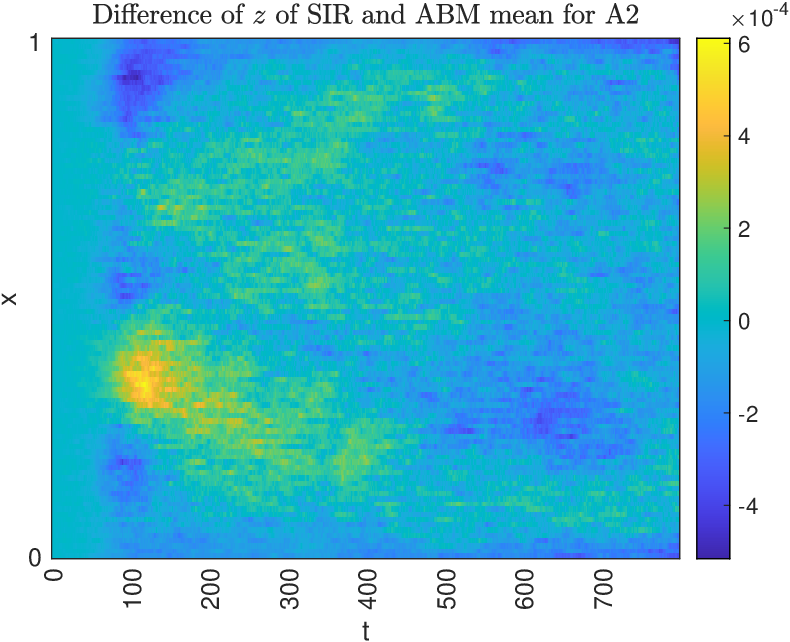}
			\end{subfigure}
				\begin{subfigure}{.4\textwidth}
					\includegraphics[width=\linewidth]{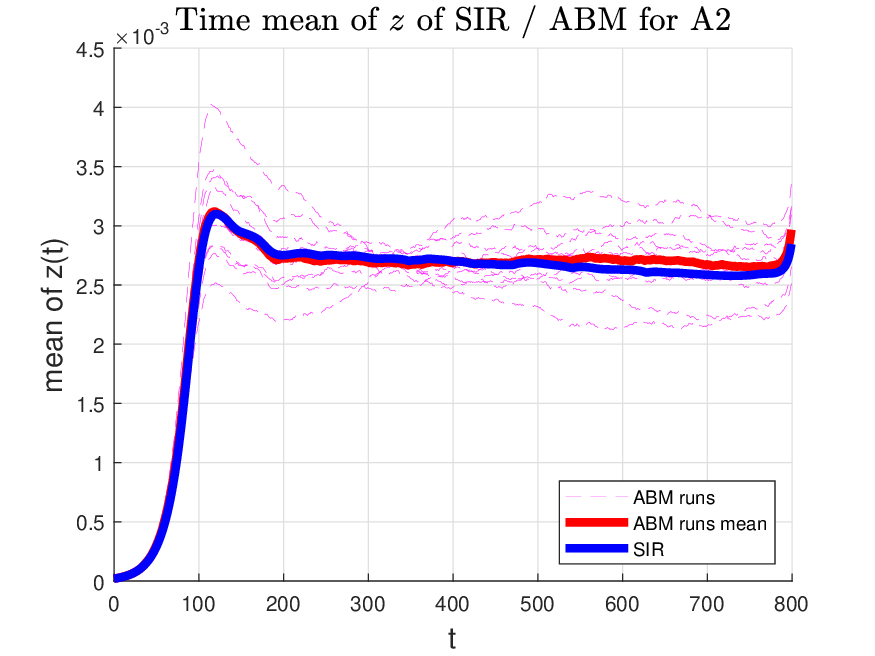}
				\end{subfigure}
				\caption{Difference between the \textit{SIR}-model to the ABM model mean (left) and temporal evolution of the spatial mean in the \textit{SIR}-model and all single runs of the ABM model, as well as their mean (right) in Simulation A2.}
							\label{sima23B}
	\end{figure}
	
\subsubsection{Simulation B1}

			\begin{figure} [H]
				\centering
				\begin{subfigure}{.4\textwidth}
					\includegraphics[width=\linewidth]{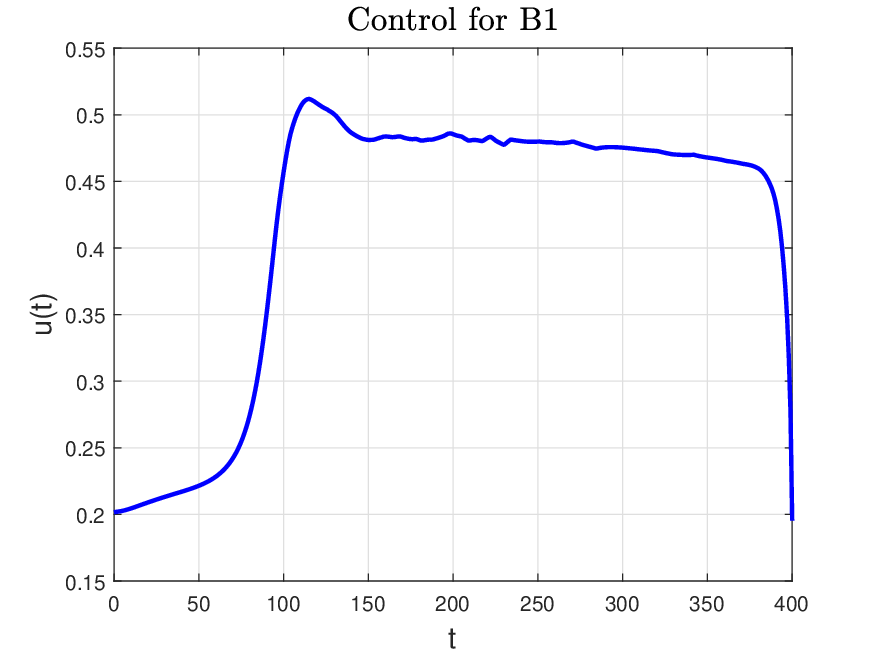}
				\end{subfigure}
				\caption{Evolution of the control in Simulation B1.}
							\label{simb13}
		\begin{figure} [H]
				\centering
			\begin{subfigure}{.4\textwidth}
				\includegraphics[width=\linewidth]{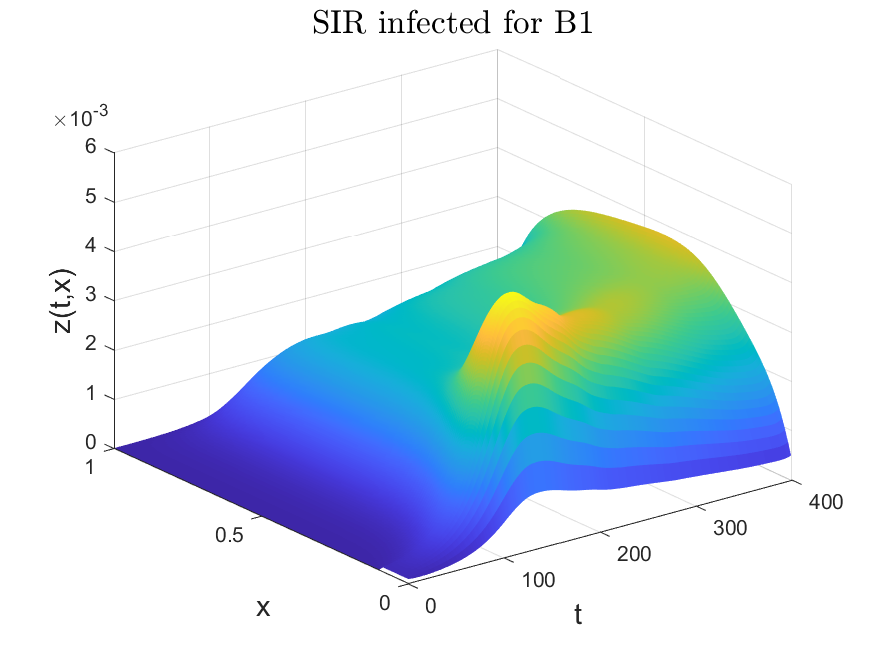}
			\end{subfigure}
				\begin{subfigure}{.4\textwidth}
					\includegraphics[width=\linewidth]{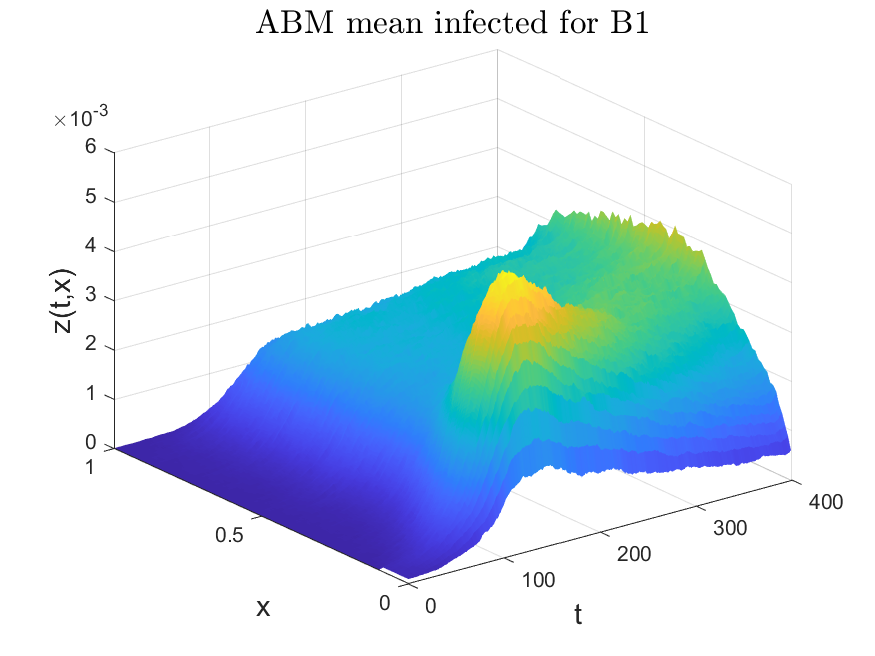}
				\end{subfigure}
				\caption{Spatio-temporal evolution of the infected in Simulation B1, on the left for the integro-differential \textit{SIR}-model, on the right for the ABM model.}
							\label{simb11}
	\end{figure}
	\end{figure}
		\begin{figure} [H]
				\centering
			\begin{subfigure}{.4\textwidth}
				\includegraphics[width=\linewidth]{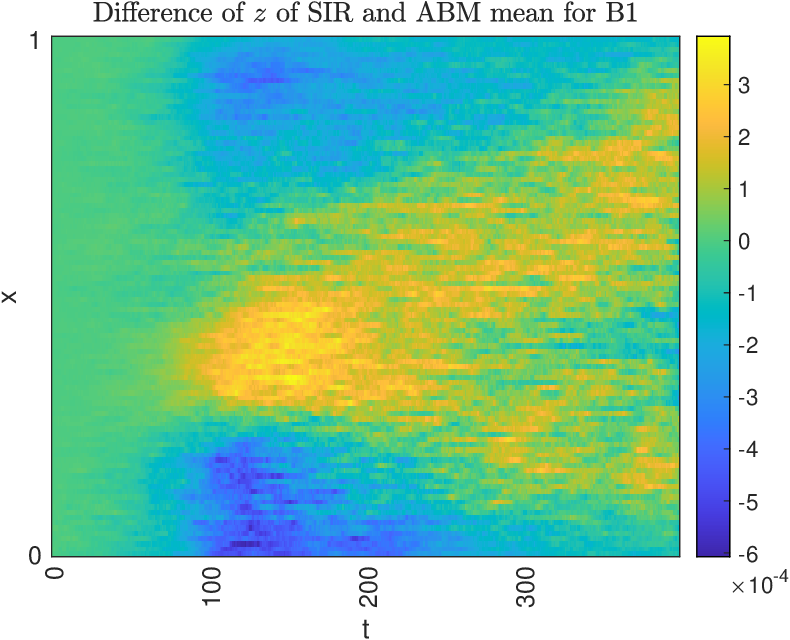}
			\end{subfigure}
				\begin{subfigure}{.4\textwidth}
					\includegraphics[width=\linewidth]{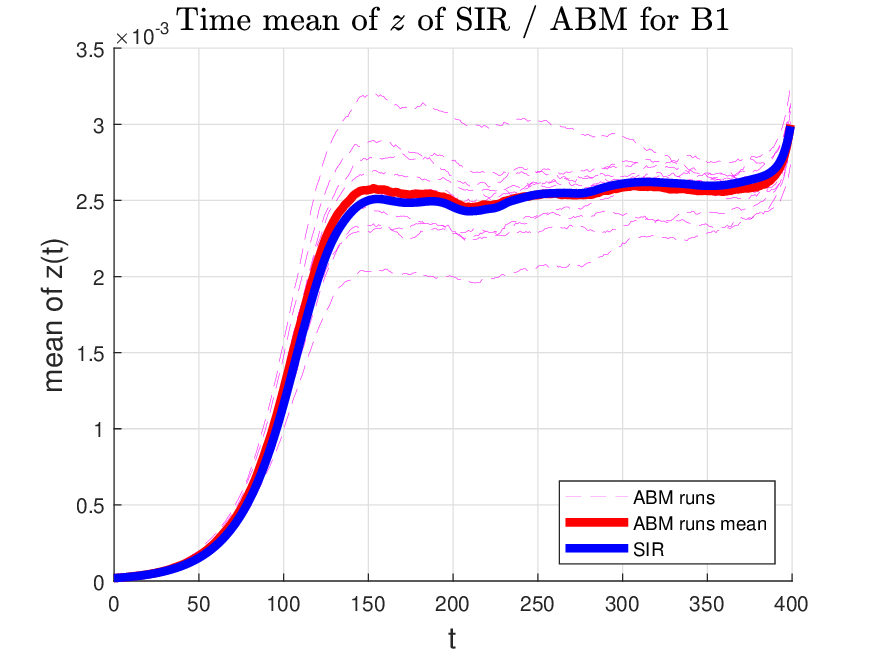}
				\end{subfigure}
				\caption{Difference between the \textit{SIR}-model to the ABM model mean (left) and temporal evolution of the spatial mean in the \textit{SIR}-model and all single runs of the ABM model, as well as their mean (right) in Simulation B1.}
							\label{simb13B}
	\end{figure}
	
	\subsubsection{Simulation B2}
	
		\begin{figure} [H]
				\centering
				\begin{subfigure}{.4\textwidth}
					\includegraphics[width=\linewidth]{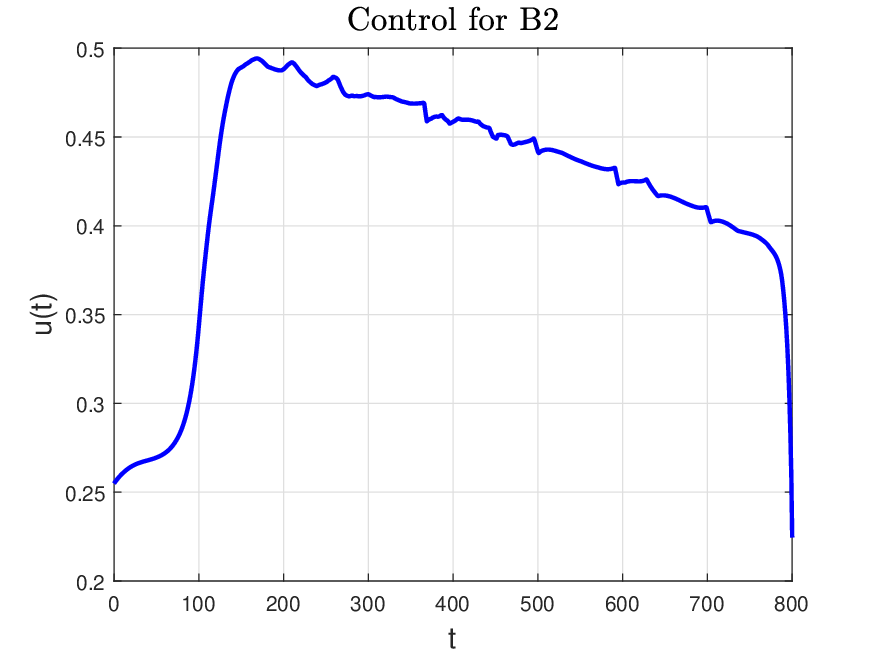}
				\end{subfigure}
				\caption{Evolution of the control in Simulation B2.}
							\label{simb23}
	\end{figure}
		\begin{figure} [H]
				\centering
			\begin{subfigure}{.4\textwidth}
				\includegraphics[width=\linewidth]{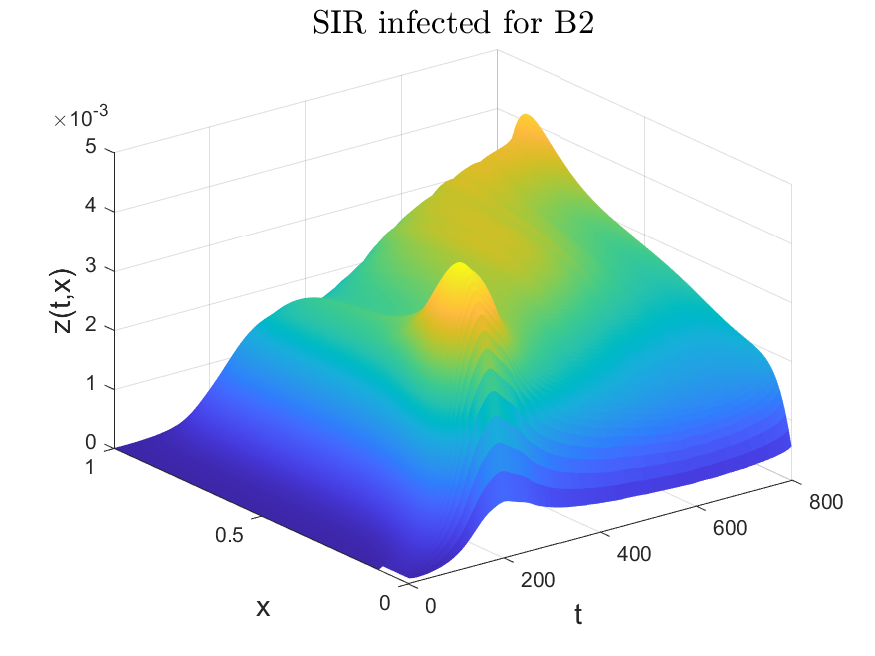}
			\end{subfigure}
				\begin{subfigure}{.4\textwidth}
					\includegraphics[width=\linewidth]{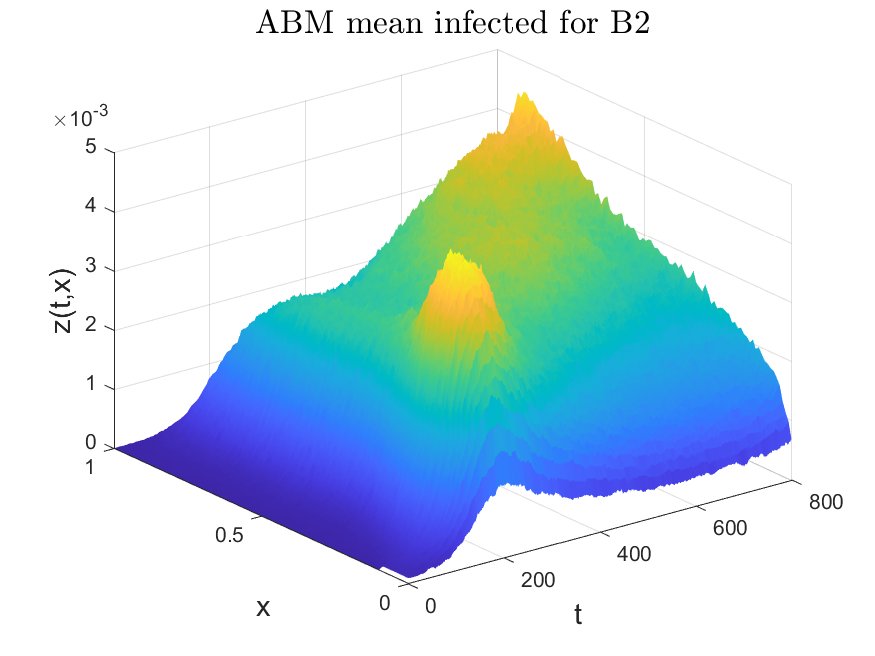}
				\end{subfigure}
				\caption{Spatio-temporal evolution of the infected in Simulation B2, on the left for the integro-differential \textit{SIR}-model, on the right for the ABM model.}
							\label{simb21}
	\end{figure}
		\begin{figure} [H]
				\centering
			\begin{subfigure}{.4\textwidth}
				\includegraphics[width=\linewidth]{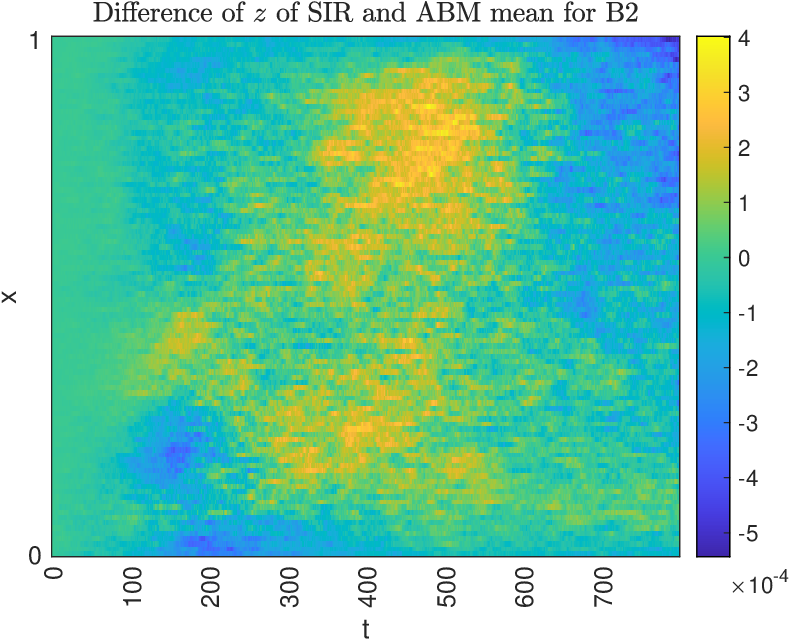}
			\end{subfigure}
				\begin{subfigure}{.4\textwidth}
					\includegraphics[width=\linewidth]{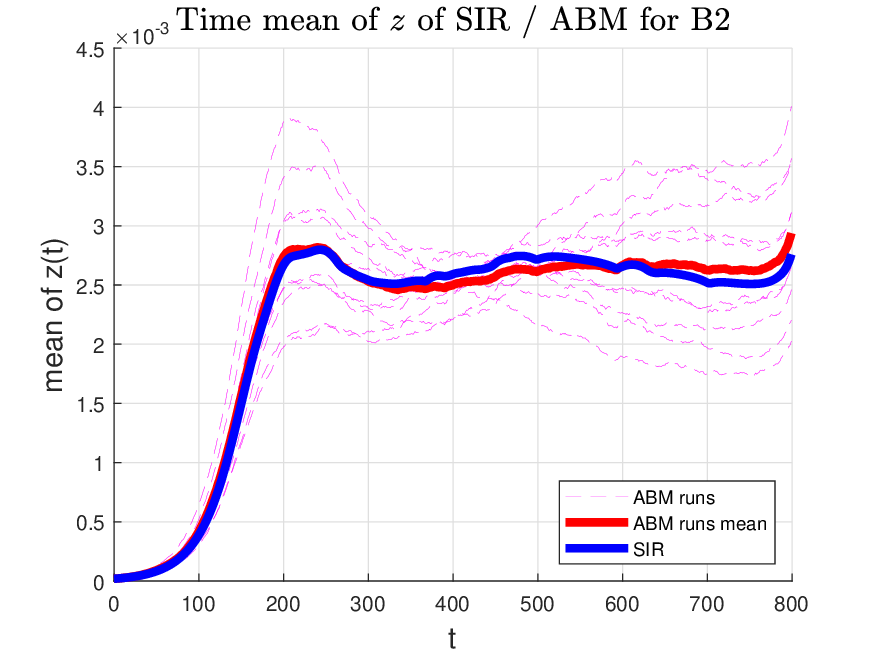}
				\end{subfigure}
				\caption{Difference between the \textit{SIR}-model to the ABM model mean (left) and temporal evolution of the spatial mean in the \textit{SIR}-model and all single runs of the ABM model, as well as their mean (right) in Simulation B2.}
							\label{simb23B}
	\end{figure}
	
	
		\subsubsection{Simulation C1}
		
	\begin{figure} [H]
				\centering
				\begin{subfigure}{.4\textwidth}
					\includegraphics[width=\linewidth]{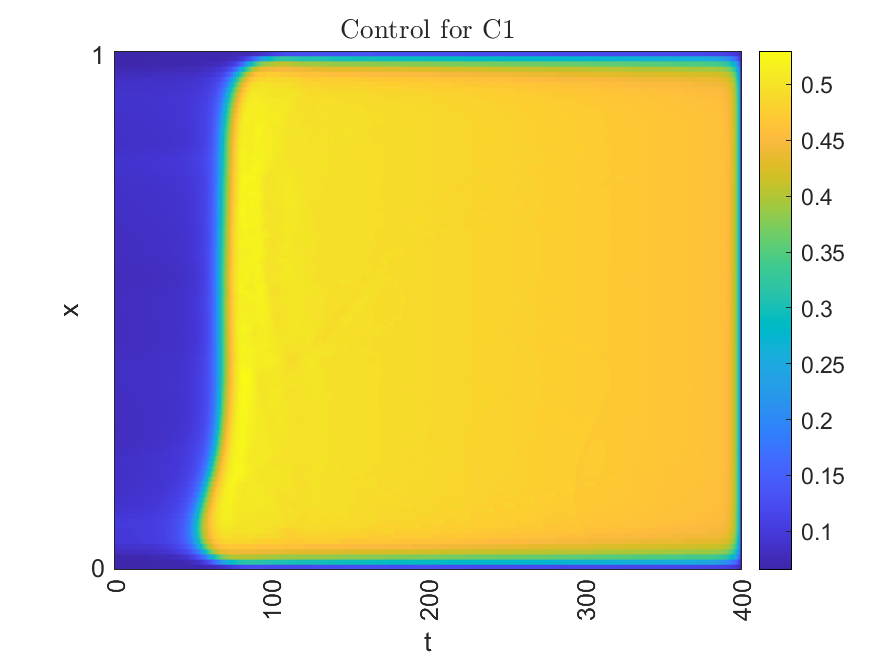}
				\end{subfigure}
				\caption{Evolution of the control in Simulation C1.}
							\label{simc13}
	\end{figure}
					\begin{figure} [H]
				\centering
			\begin{subfigure}{.4\textwidth}
				\includegraphics[width=\linewidth]{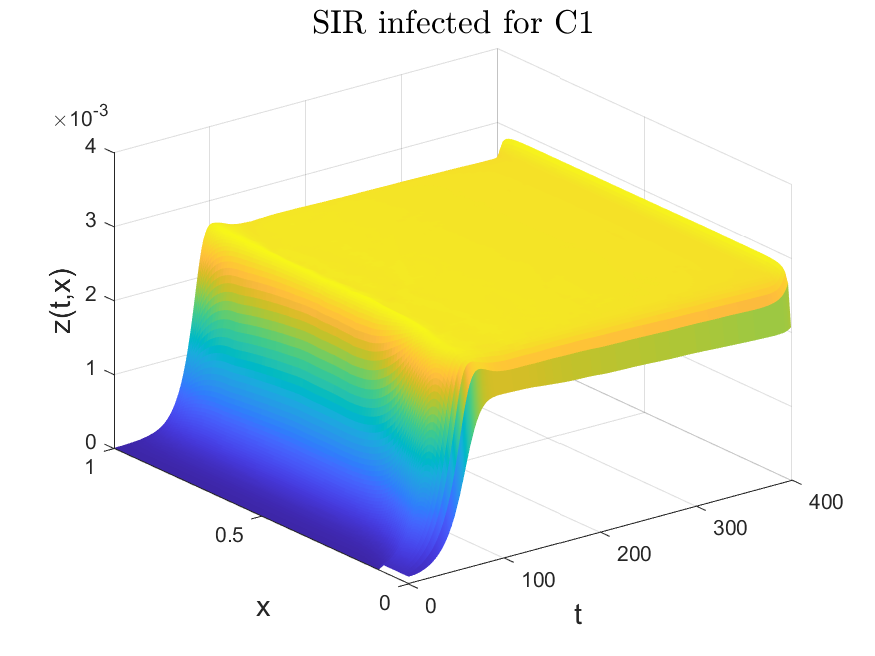}
			\end{subfigure}
				\begin{subfigure}{.4\textwidth}
					\includegraphics[width=\linewidth]{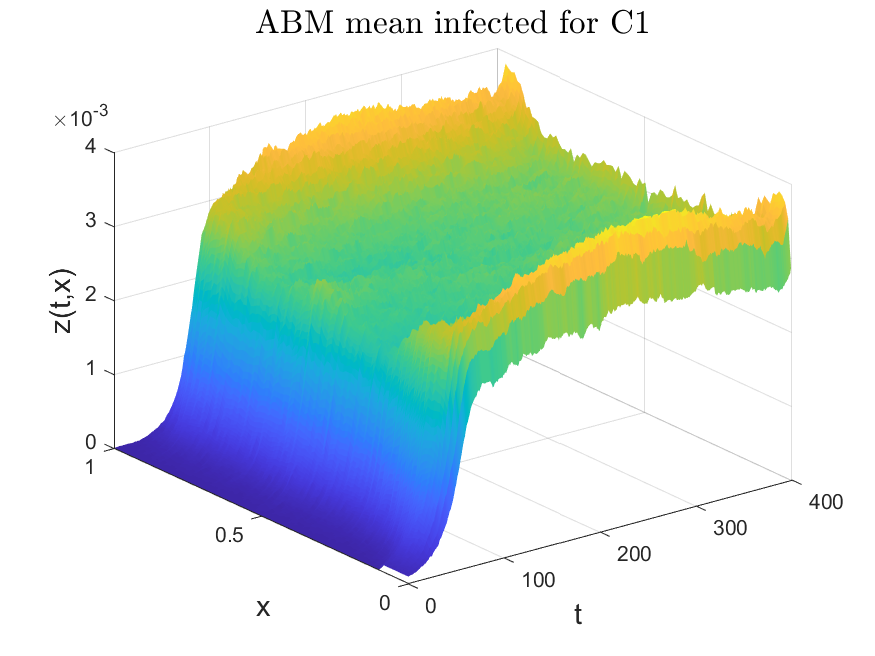}
				\end{subfigure}
				\caption{Spatio-temporal evolution of the infected in Simulation C1, on the left for the integro-differential \textit{SIR}-model, on the right for the ABM model.}
							\label{simc11}
	\end{figure}

	\begin{figure} [H]
				\centering
			\begin{subfigure}{.4\textwidth}
				\includegraphics[width=\linewidth]{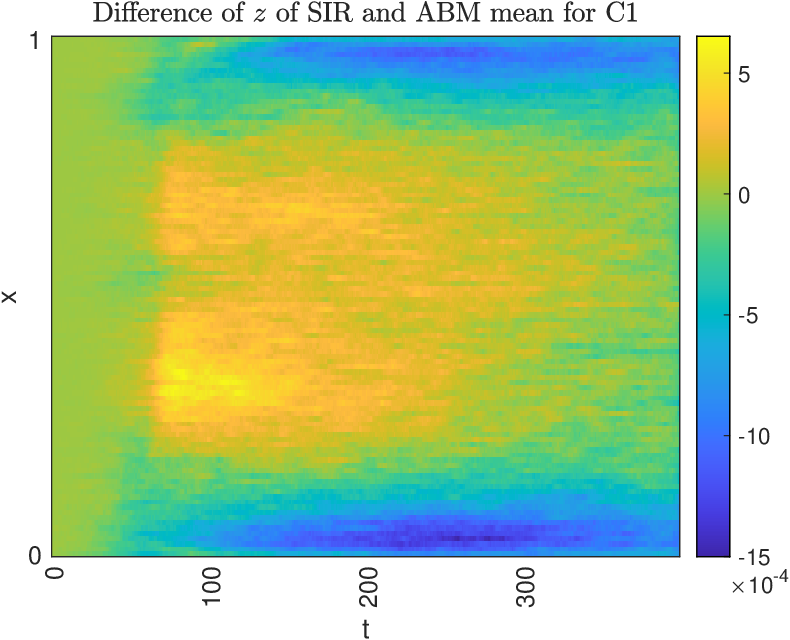}
			\end{subfigure}
				\begin{subfigure}{.4\textwidth}
					\includegraphics[width=\linewidth]{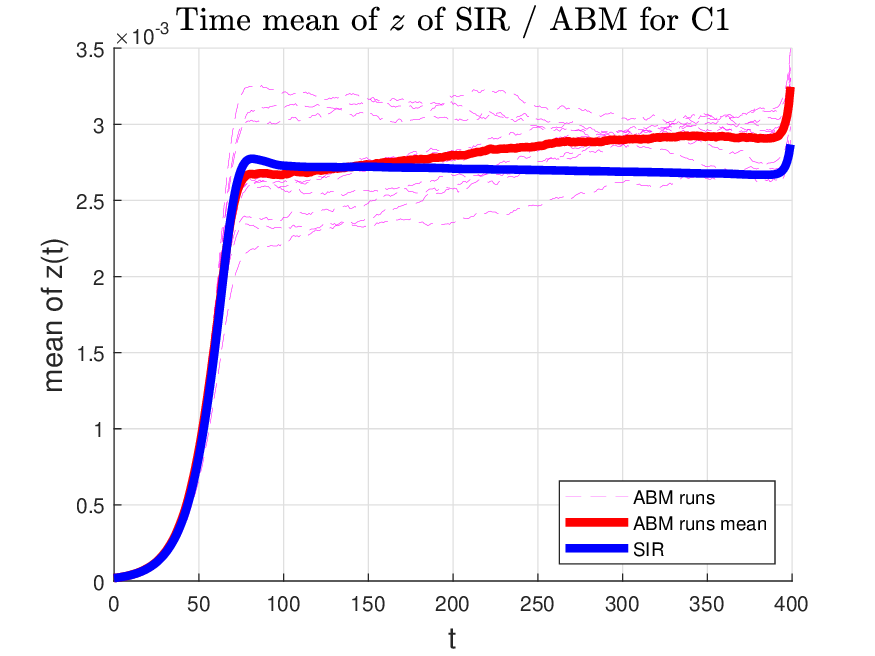}
				\end{subfigure}
			\caption{Difference between the \textit{SIR}-model to the ABM model mean (left) and temporal evolution of the spatial mean in the \textit{SIR}-model and all single runs of the ABM model, as well as their mean (right) in Simulation C1.}
							\label{simc13B}
	\end{figure}
	
\subsubsection{Simulation C2}

			\begin{figure} [H]
				\centering
				\begin{subfigure}{.4\textwidth}
					\includegraphics[width=\linewidth]{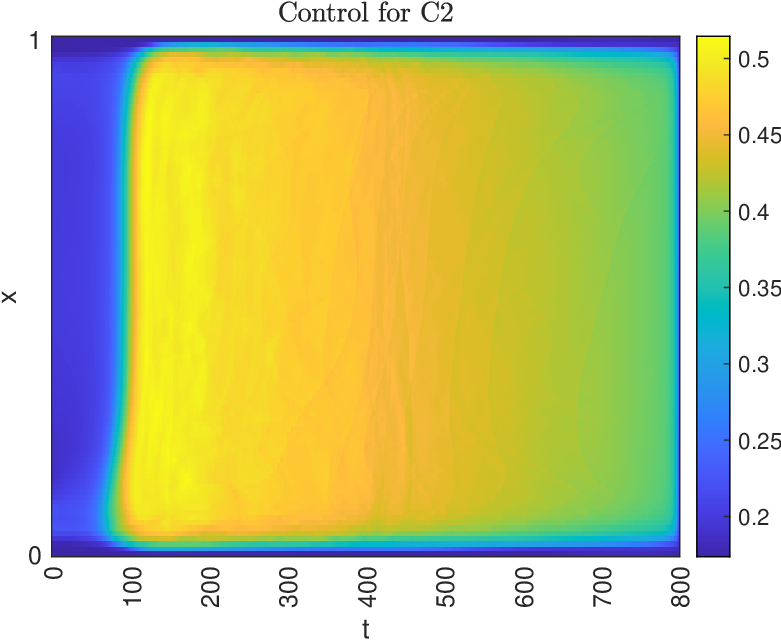}
				\end{subfigure}
	
				\caption{Evolution of the control in Simulation C2.}
							\label{simc23}
	\end{figure}
		\begin{figure} [H]
				\centering
			\begin{subfigure}{.4\textwidth}
				\includegraphics[width=\linewidth]{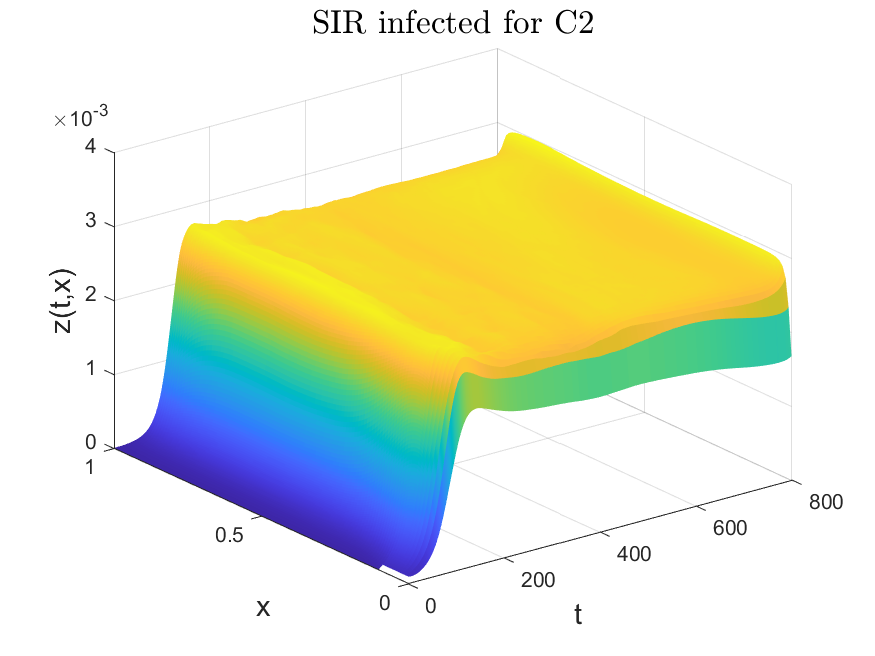}
			\end{subfigure}
				\begin{subfigure}{.4\textwidth}
					\includegraphics[width=\linewidth]{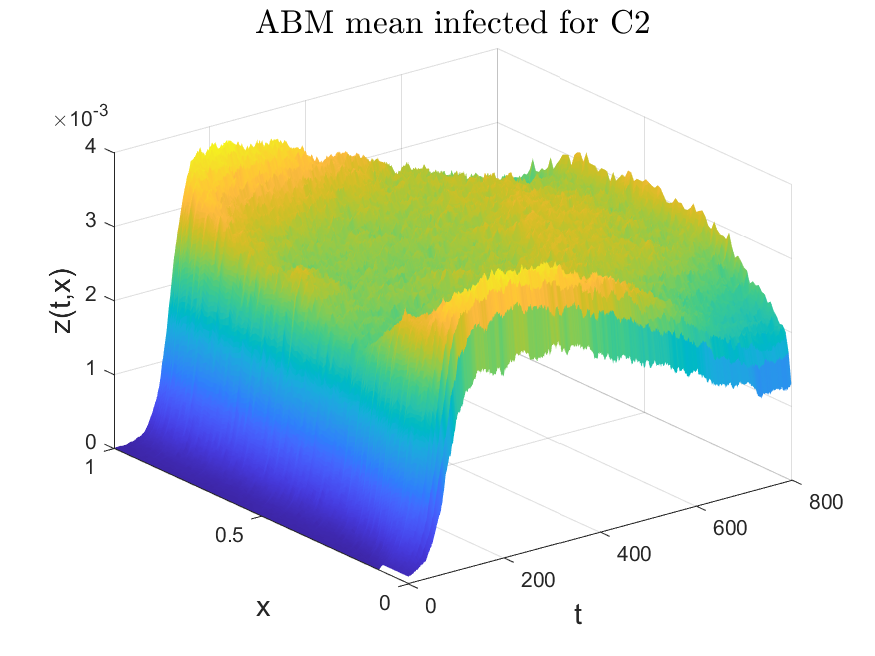}
				\end{subfigure}
				\caption{Spatio-temporal evolution of the infected in Simulation C2, on the left for the integro-differential \textit{SIR}-model, on the right for the ABM model.}
							\label{simc21}
	\end{figure}
		\begin{figure} [H]
				\centering
			\begin{subfigure}{.4\textwidth}
				\includegraphics[width=\linewidth]{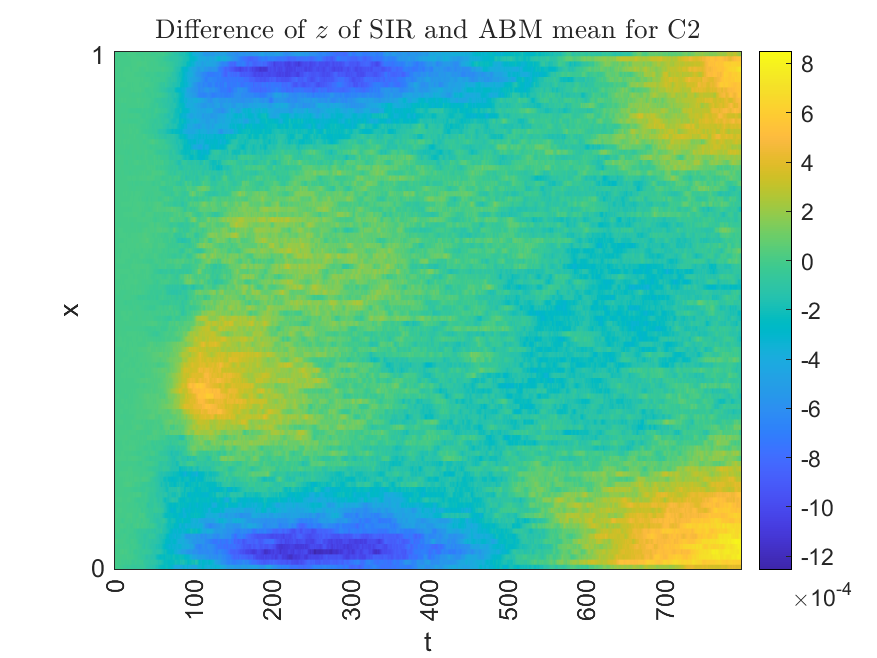}
			\end{subfigure}
				\begin{subfigure}{.4\textwidth}
					\includegraphics[width=\linewidth]{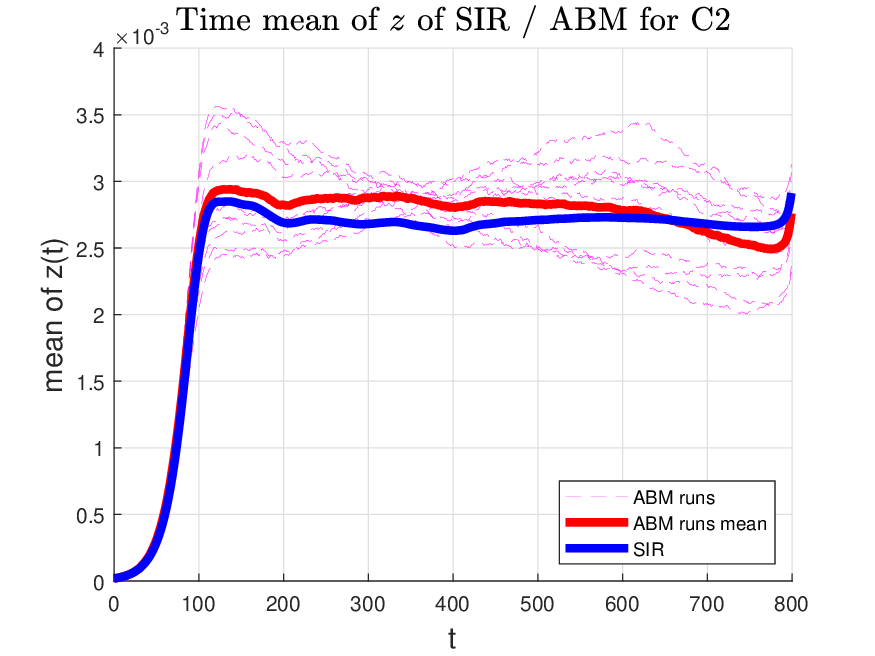}
				\end{subfigure}
				\caption{Difference between the \textit{SIR}-model to the ABM model mean (left) and temporal evolution of the spatial mean in the \textit{SIR}-model and all single runs of the ABM model, as well as their mean (right) in Simulation C2.}
							\label{simc23B}
	\end{figure}

		\subsubsection{Simulation D1}
		
		\begin{figure} [H]
				\centering
				\begin{subfigure}{.4\textwidth}
					\includegraphics[width=\linewidth]{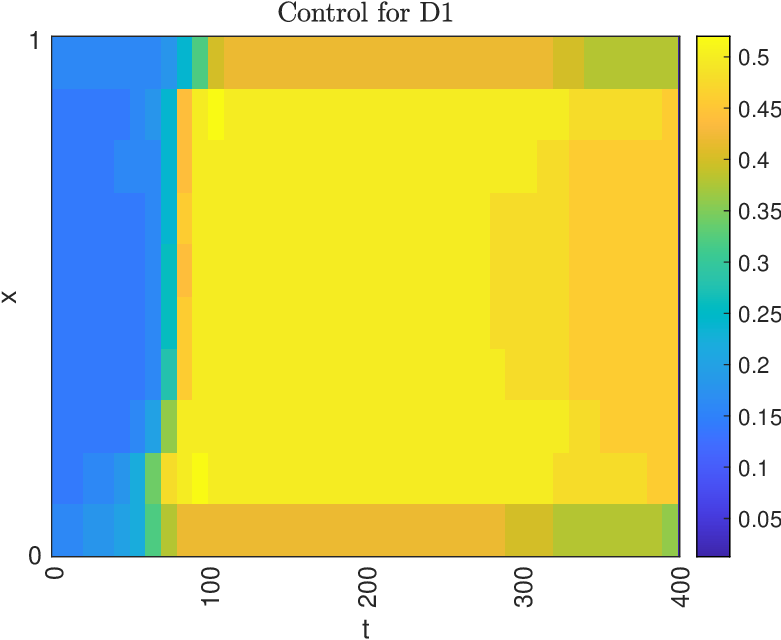}
				\end{subfigure}
				
				\caption{Evolution of the control in Simulation D1.}
							\label{simco13}
	\end{figure}
					\begin{figure} [H]
				\centering
			\begin{subfigure}{.4\textwidth}
				\includegraphics[width=\linewidth]{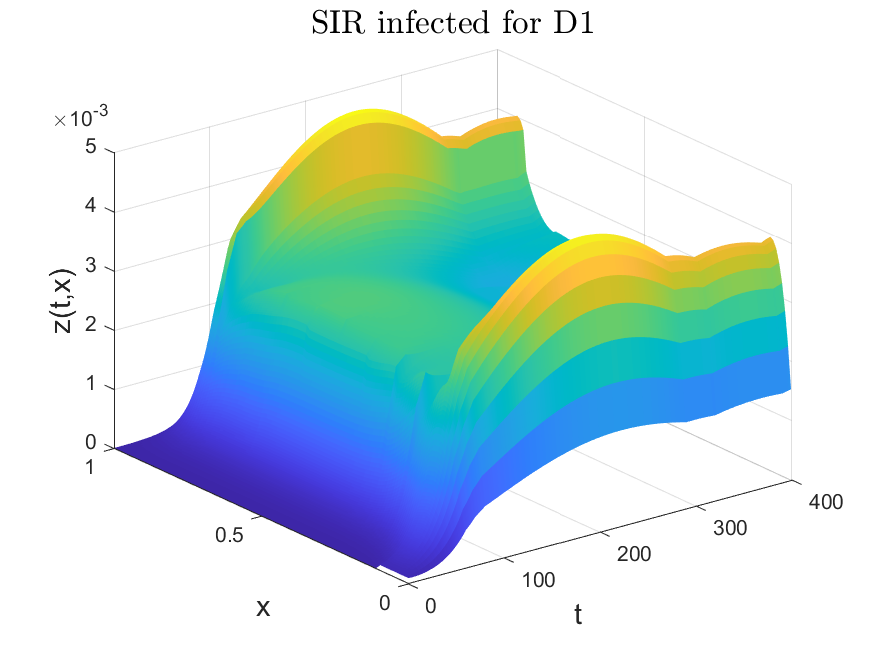}
			\end{subfigure}
				\begin{subfigure}{.4\textwidth}
					\includegraphics[width=\linewidth]{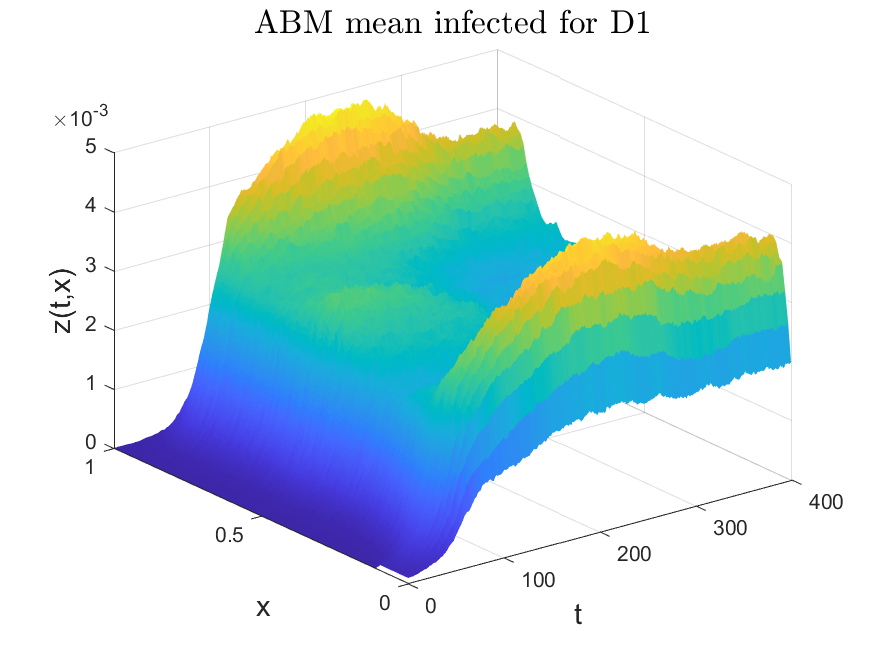}
				\end{subfigure}
				\caption{Spatio-temporal evolution of the infected in Simulation D1, on the left for the integro-differential \textit{SIR}-model, on the right for the ABM model.}
							\label{simco11}
	\end{figure}
		\begin{figure} [H]
				\centering
			\begin{subfigure}{.4\textwidth}
				\includegraphics[width=\linewidth]{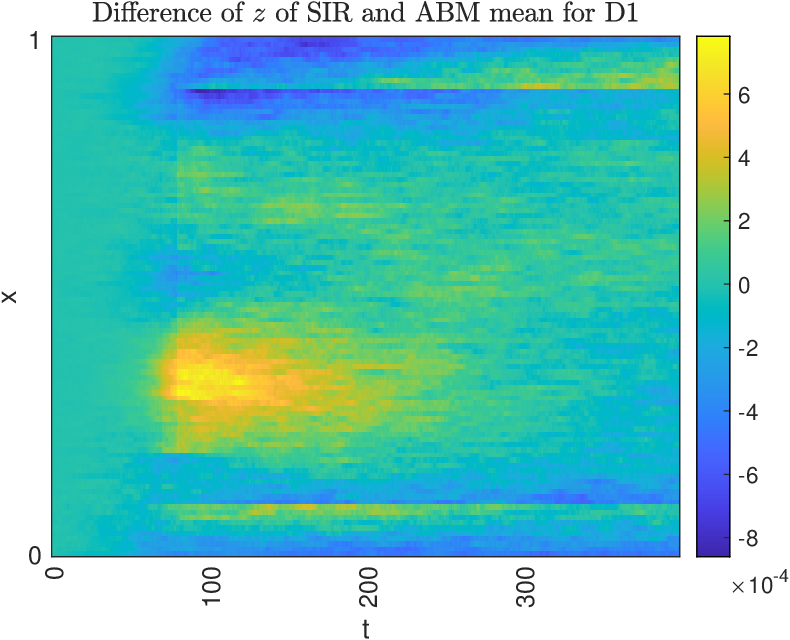}
			\end{subfigure}
				\begin{subfigure}{.4\textwidth}
					\includegraphics[width=\linewidth]{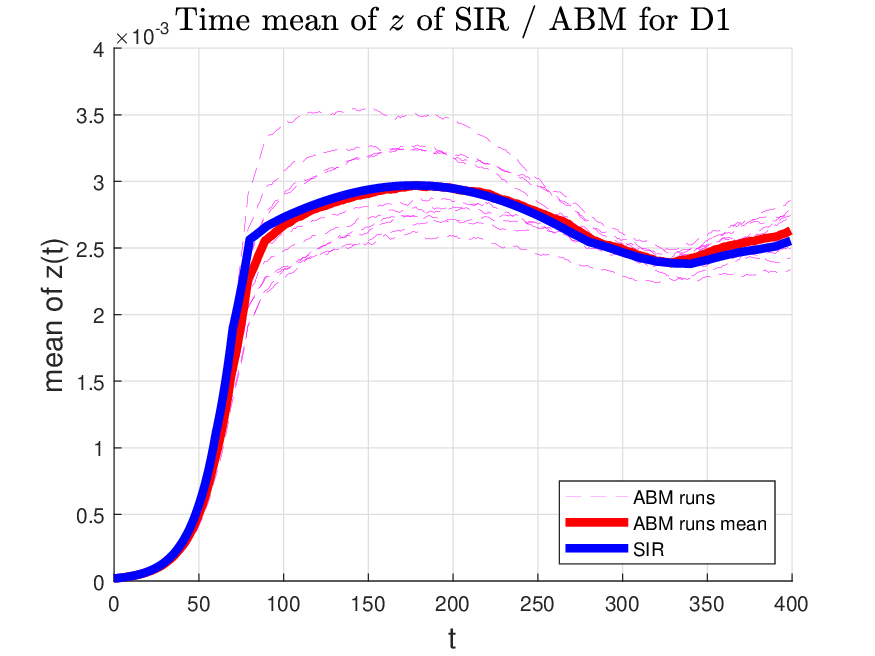}
				\end{subfigure}
			\caption{Difference between the \textit{SIR}-model to the ABM model mean (left) and temporal evolution of the spatial mean in the \textit{SIR}-model and all single runs of the ABM model, as well as their mean (right) in Simulation D1.}
							\label{simco13B}
	\end{figure}

\subsubsection{Simulation D2}

		\begin{figure} [H]
				\centering
				\begin{subfigure}{.4\textwidth}
					\includegraphics[width=\linewidth]{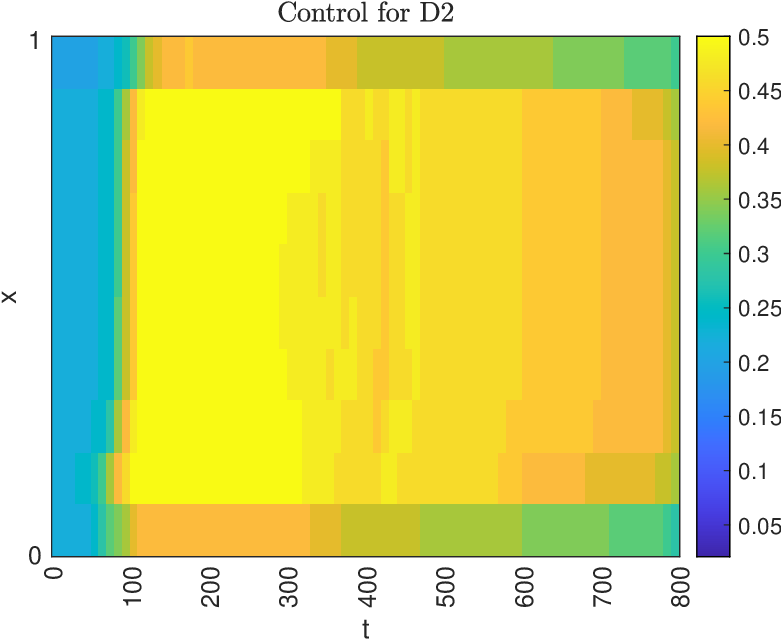}
				\end{subfigure}

				\caption{Evolution of the control in Simulation D2.}
							\label{simco23}
	\end{figure}
		\begin{figure} [H]
				\centering
			\begin{subfigure}{.4\textwidth}
				\includegraphics[width=\linewidth]{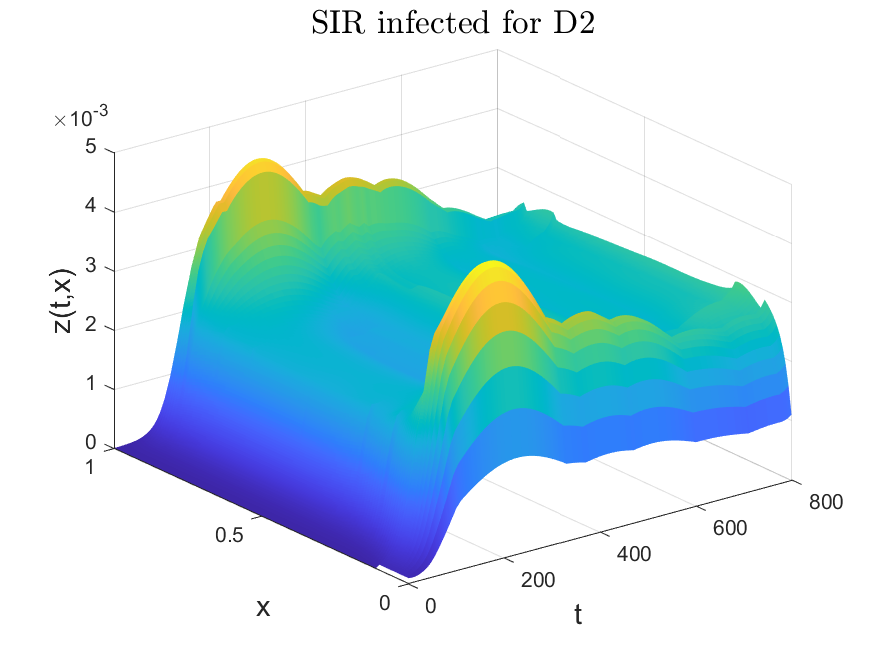}
			\end{subfigure}
				\begin{subfigure}{.4\textwidth}
					\includegraphics[width=\linewidth]{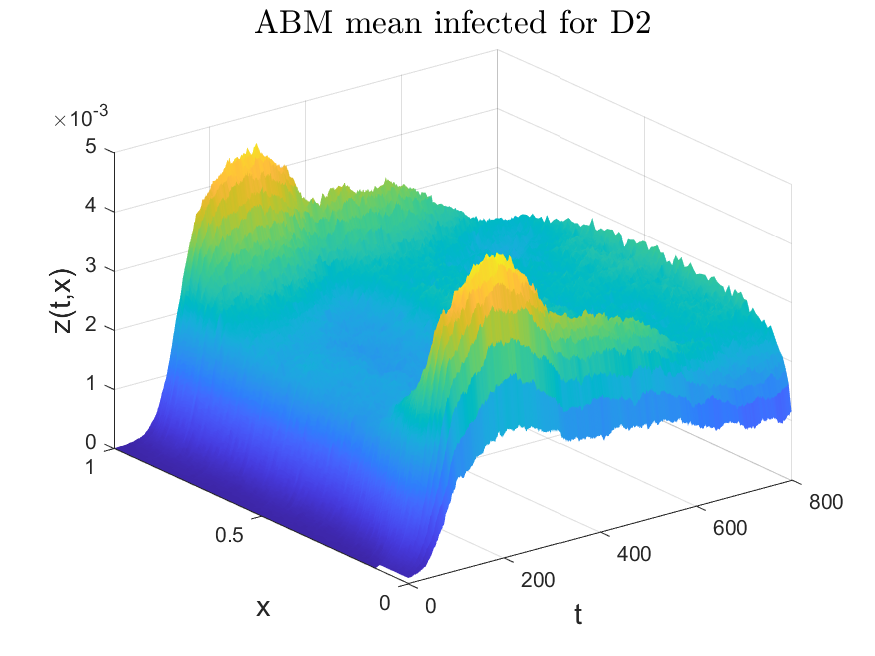}
				\end{subfigure}
				\caption{Spatio-temporal evolution of the infected in Simulation D2, on the left for the integro-differential \textit{SIR}-model, on the right for the ABM model.}
							\label{simco21}
	\end{figure}
		\begin{figure} [H]
				\centering
			\begin{subfigure}{.4\textwidth}
				\includegraphics[width=\linewidth]{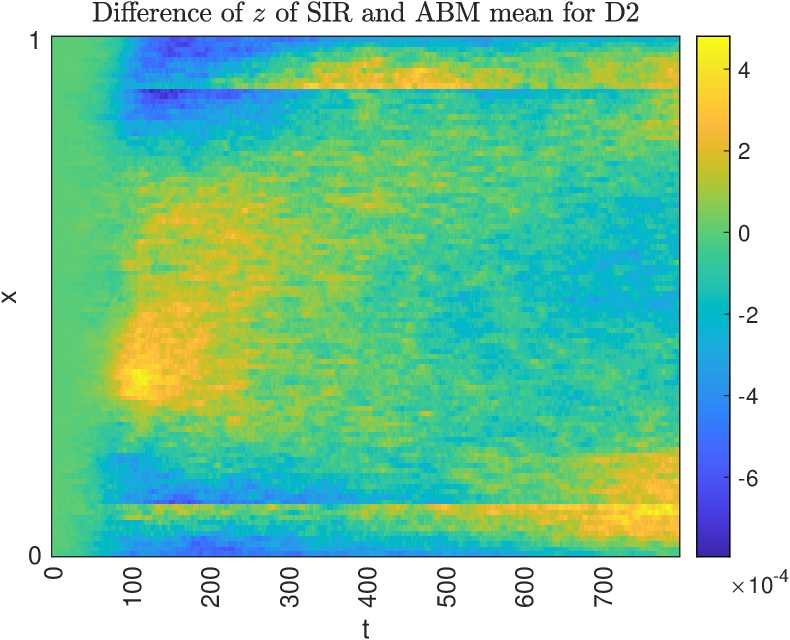}
			\end{subfigure}
				\begin{subfigure}{.4\textwidth}
					\includegraphics[width=\linewidth]{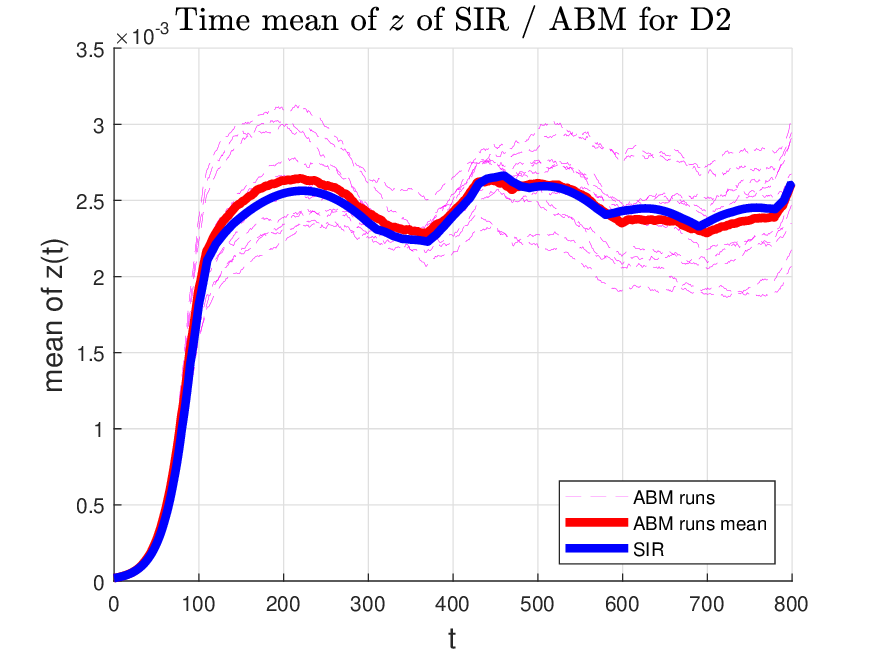}
				\end{subfigure}
				\caption{Difference between the \textit{SIR}-model to the ABM model mean (left) and temporal evolution of the spatial mean in the \textit{SIR}-model and all single runs of the ABM model, as well as their mean (right) in Simulation D2.}
							\label{simco23B}
	\end{figure}
	\newpage
	
\subsection{Observations for the integro-differential model}

In simulation A1 (cf. Figs. \ref{sima13},\ref{sima11},\ref{sima13B}), using an only time-dependent control function and homogeneous initial conditions, we see that after light restrictions causing raising numbers of infectives, a sudden increase in lockdown restrictions causes falling numbers. The control only rises as late as possible in order not to surpass or get close to $z_{\max}$, and thereon remains relatively constant. Raising to levels at or slightly above 0.5, it manages to contain the disease spread due to $\mathcal R_0 = 2 $. Later on, the effective reproduction rate lowers as there is a higher amount of recovered which are assumed to not get infected again in this simplified model. At the end, roughly 10 $\%$ of the total population is infected or recovered, homogeneously spread on the whole spatial domain. The drop of the control function towards the end can be explained as the system regulates the epidemics such that the maximally allowed rates are just slightly missed at the end. This observation is in fact independent of the chosen duration or simulation. The convergence of the target function $J(u)$ is exemplary shown for Sim. A1 (Fig. \ref{sim_targ}) and looks similar for all other simulations. The results of Sim. A2 (cf. Figs. \ref{sima23},\ref{sima21},\ref{sima23B}) are comparable to Sim. A1, yet due to different values of $\eta$ and $\omega$, the control reduces more quickly in a quite linear fashion after rising above 0.5. As a result, the total amount of recovered and infectives is roughly equal to 20 $\%$ at the end.

In simulations B1 (cf. Figs. \ref{simb13},\ref{simb11},\ref{simb13B}) and B2 (cf. Figs. \ref{simb23},\ref{simb21},\ref{simb23B}), an only time-dependent control function and an inhomogeneous initial condition are used. We observe similar optimal control functions in Sims. B1 and B2 as in A1 and A2, respectively, resulting in a spatially delimited peak of infections slightly propagating in time. Due to generally higher infection cases in Sim. B2, there is another peak at the end of the observed  time interval. The share of recovered again reaches values around  10 $\%$ for B1, with a peak close to the boundary from which the disease was important, and 20 $\%$ for B2, with homogeneously distributed values across the spatial domain.

Simulations C1 (cf. Figs. \ref{simc13},\ref{simc11},\ref{simc13B}) and C2 (cf. Figs. \ref{simc23},\ref{simc21},\ref{simc23B}) feature a space-time-dependent control function and an inhomogeneous initial condition. While the spatially averaged behaviour of the control function in C1 and C2 is similar as in B1 and B2, respectively, a more or less slight spatial 'propagation' of the control is visible. This adaptive behaviour allows for the control term to never surpass 0.5, resulting in less effort and thus a lower target function. The share of recovered again reaches values around  10 $\%$ for C1, and 20 $\%$ for C2, both with homogeneously distributed values across the spatial domain.

Finally, simulations D1 (cf. Figs. \ref{simco13},\ref{simco11},\ref{simco13B}) and D2 (cf. Figs. \ref{simco23},\ref{simco21},\ref{simco23B}) feature a piecewise constant space-time-dependent control function and an inhomogeneous initial condition. In those simulations, as to be expected, the control is similar to the one in the continuous simulations C1 and C2. However, using the starting value as the control for the next (10) days, causes higher infection rates in the initial phase of the disease, such that it can be said that globally the control has to be slightly larger as in the continuous simulations. The share of recovered again reaches values around  10 $\%$ for D1, and 20 $\%$ for D2, yet features significant peaks at both boundaries.

\subsection{Comparison with the agent-based model}

Tab. \ref{tab:sim2} lists the values of the target function of all simulations according to eqns. \eqref{eq_hard} and \eqref{eq_hard1}. Additionally, the target function of a model without any control measures, i.e. $u(t) \equiv 0$ or $u(t,x)\equiv 0$, is listed, showing significant improvement in the target function for simulations C1 and C2, while D1 and D2 were still reasonably good in reducing the cost function values despite their restrictions.
\begin{table}[H]
    \centering
        \caption{Target function values for the various simulations, according to eqns. \eqref{eq_hard} and \eqref{eq_hard1}.}
    \label{tab:sim2}
    \begin{tabular}{c|c|c|c}
    Simulation    & $J(u\equiv0)$& $J(u^*)_\text{sir}$& $J(u^*)_\text{abm}$
    \\
    \hline
    A1 & 132.4 &  31.9& 29.5
    \\
    A2 &\,\,\,32.9 &  13.4 & 12.6
    \\
    B1 & 132.5 &  40.6 & 39.0
    \\
    B2 & \,\,\,32.9 &   19.9& 19.3
    \\
    C1 & 132.5 &   10.1 & 12.6
    \\
    C2 & \,\,\,32.9 &   \,\,\,8.5& \,\,\,9.0
    \\
    D1 & 132.5 &   19.8 & 21.5
    \\
    D2& \,\,\,32.9 &   12.2 & 12.5
    \end{tabular}
\end{table}
If we compare the results of the target function, it shows that by using the optimal control, it was possible to reduce the target function compared to $u\equiv 0$ by a factor depending on the chosen simulation and parameter values of $\eta$ and $\omega$. As to be expected, the best results for $J(u)$ were found in the space-dependent, yet continuous control. Comparing the results for the target function for the \textit{SIR}-model and the ABM model, we see that there are only minor differences in the outcome, mainly in C1 and C2 and there especially close to the boundaries. Those can be explained by the stochastic nature of the ABM model of which also not all features can be adapted to the integro-differential model.

\section{Discussion and Outlook}

In this work we have presented an integro-differential \textit{SIR}-model, have proved several theoretical properties (including uniqueness of the solution) and provided setups in order to apply optimal control on the transmission of the disease. The results were compared to the ABM model and are overall very much accordant. This has an interesting  consequence; the \textit{SIR}-model is   'cheap' to compute, compared to the computationally 'expensive' ABM model, for which an optimal control is hardly possible. Thus, by optimization of the \textit{SIR}-model we are now able to find a good proxy for the ABM model. While it will not be possible to reproduce the results perfectly, the averages of both models are very similar and match very well in most simulations except the space-and time-dependent continuous control. The results always remain within the designed range $[z_\text{min},z_\text{max}]$, such that the healthcare capacities are not overloaded, even in the model with piecewise constant values for $u$ that is less flexible to quickly raising infection numbers. The system still remains in the range of stochastic fluctuations of the ABM model.

As future work in this context aims to apply the optimal control proxy on ABM models, we aim to use multidimensional problems, i.e., a spatially 2D-problem which actually represents a more realistic approach for entire countries (like Poland in the Warsaw model). Another interesting application of this integro-differential model lies in models of age-structure, where the parameter $x$ is interpreted as the age, and we can transform a discrete contact matrix for age cohorts into a kernel function. While the integro-differential model can be enhanced and made more realistic by adding more compartments and parameters, e.g., vaccination and household interactions (corresponding to the previously unused $k_0$) or an \textit{SEIR}-model, it is advisable to keep the model as simple as possible in order to maintain computational effectivity. However, using the parameters and knowledge we have gained from this work, we can implement the ABM model including households, and perform parameter estimation to find reasonable values for $k_0$ in the integro-differential \textit{SIR}-model.

\bibliography{output.bib}

\section*{Declarations}







\subsection*{Competing interests}

The authors declare that there exist no competing interests.

\subsection*{Funding}

This collaboration between the groups in Koblenz and in Poland was funded by the DAAD--NAWA joint project ''MultiScale Modelling and Simulation for Epidemics''--MSS4E, DAAD project number: 57602790, NAWA grant number: PPN/BDE/2021/1/00019/U/DRAFT/00001. 

The ICM model was developed as part of the ''ICM Epidemiological Model development'' project, funded by the Ministry of Science and Higher Education of Poland with grants 51/WFSN/2020T and 28/WFSN/2021 to the University of Warsaw.

\subsection*{Acknowledgements}

We would like to thank the ICM Epidemiological Model team members for development of a new software engine for agent-based simulations (pDyn2) that was used in generation of part of the results of this work.

Tyll Krüger thanks Wroclaw University of Science and Technology for providing the necessary infrastructure and scientific environment for several meetings of the authors in Wroclaw within the NAWA project.

\subsection*{Authors' contributions}

All authors contributed equally and reviewed the manuscript.

\subsection*{Availability of data}

We used several of the underlying MATLAB codes for multiple purposes and thus believe the open presentation of the codes would distract readers. However, the codes, datasets used, and/or analyzed during the current study are available from the corresponding author on reasonable request. 

\newpage

\section*{Appendix A - Existence and uniqueness of solutions of the SIS-model}

The proof of uniqueness of solutions of the \textit{SIR}-model made some theory redundant that was previously performed for the \textit{SIS}-model. However, we decided to publish the most relevant parts of it in this appendix.

Consider an \textit{SIS}-version of eqns.~\eqref{E:SIR}, wlog setting $\beta=\gamma=1$. Then, this equation reads in fix point form as
\begin{align}
    \label{E:SIS}
    z(x) = (1-z(x))\cdot \int_0^1 z(y)\cdot k({x-y})\, dy .
\end{align}
The trivial disease free equilibrium of the model~\eqref{E:SIS} is given by $z\equiv 0$. Before dealing with the question, whether there might exist other, non-trivial equilibria, we will first prove
\begin{lemma}
\label{L:sym}
Let $z:[0,1]\to [0,1]$ be a continuous solution of the fix point equation~\eqref{E:SIS}. Then $z$ is symmetric, i.e.~$z(x)=z(1-x)$ for all $x\in [0,1]$.
\end{lemma}

\begin{proof}
It holds that
\begin{align}
    z(1-x) = \lb 1-z(1-x)\rb \, \int_0^1 z(y)\, k({1-x-y})\, dy.
\intertext{After substituting $s=1-y$ in the integral we arrive at}
    z(1-x) = \lb 1-z(1-x)\rb \, \int_0^1 z(1-s)\, k({s-x})\, ds.
\intertext{Hence, $\tilde{z}(x):=z(1-x)$ solves}
    \tilde{z}(x) = \lb 1- \tilde{z}(y)\rb \, \int_0^1 \tilde{z}(y)\, k({x-y})\, dy\,.
\end{align}
This implies that $\tilde{z}$ solves the same equation as $z$. Also, consider the following
\end{proof}
\begin{lemma}
\label{L:1}
Let $z:[0,1]\to [0,1]$ be a continuous solution of the fix point equation~\eqref{E:SIS}. Then either $z(x)=0$ for all $x\in [0,1]$ or $z(x)>0$ for all $x\in [0,1]$.
\end{lemma}
\begin{proof}
Assume that $z(x)\not\equiv 0$. Then there exists $x_0 \in [0,1]$ such that $z(x_0)=0$ and wlog $z(x_0+\delta)>0$ for all $0<\delta<\delta_0$. Due to continuity of $k$ and $k_0>0$, there exists a $\delta_1>0$, such that $k(r)>0$ for all $0\le r <\delta_1$. Thus, it follows
\begin{align}
    z(x_0) &= (1-z(x_0))\int_0^1 z(y)\, k(\abs{x_0-y})\, dy
        \\&\ge (1-z(x_0)) \int_0^{\min(\delta_0,\delta_1)} z(x_0+r)\, k(r)\, dr > 0.
\end{align}
\end{proof}
Next, we define $\chi[z](x):= \int_0^1 z(y)\ k(\abs{x-y})\, dy$. Then the fix point equation~\eqref{E:SIS} reads as 
\begin{align}
   z= \Phi[z] := \frac{\chi[z]}{1+\chi[z]}\;. 
\end{align}

We show, that any non--trivial solution of~\eqref{E:SIS} has to satisfy a--priori bounds
\begin{lemma}
\label{L:2}
Let $z:[0,1]\to [0,1]$ be a non--trivial solution of $z=(1-z)\chi[z]$. Then for all $x\in [0,1]$ we find
\begin{align}
    \frac{k_1-1}{k_1} \le z(x) \le \frac{K-1}{K}.
\end{align}
\end{lemma}
\begin{proof}
To show the upper bound: By $z(x)>0$, we know $\chi[z](x)>0$ and therefore $z=\chi[z]/(1+\chi[z])<1$. If $z(x)<1$, then $\chi[z](x)<K$ and due to $\chi[z]<K$ we get $z(x)<{K}/(1+K)$. Iterating this, for all $n\in\N$ it holds 
\begin{align}
   z(x)<\frac{K^n}{1+K+\dots + K^n} = \frac{(K-1) K^n}{K^{n+1}-1} = \frac{K-1}{K-K^{-n}}
\end{align}
 Passing to the limit $n\to \infty$ yields $z(x)\le (K-1)/{K}$. On the other hand, if $z$ is non--trivial, then due to Lemma~\ref{L:1}, there exists an $\varepsilon>0$ such that $z(x)>\varepsilon$ for all $x$. Hence $\chi[z](x)>\varepsilon k_1$ and using $z=\chi[z]/(1+\chi[z])$ as well as the monotonicity of the function $x\mapsto x/(1+x)$ we get $z(x)>{\varepsilon k_1}/({1+\varepsilon k_1})$. Iterating again between estimates for $\chi[z]$ and $z$ finally yields
\begin{align}
  z(x) > \frac{\varepsilon k_1^n}{1+\varepsilon k_1+\dots +\varepsilon k_1^n} 
    = \frac{\varepsilon k_1^n}{1+\frac{\varepsilon k_1 (k_1^n -1)}{k_1-1}} 
    = \varepsilon \lsb k_1^{-n} + \frac{\varepsilon k_1}{k_1-1}(1-k_1^{-n})\rsb
\end{align}
for all $n\in \N$ and after passing to the limit $n\to \infty$ we get $z(x) \ge ({k_1-1})/{k_1}$.
\end{proof}
As an immediate consequence of Lems. \ref{L:1} and \ref{L:2} we obtain
\begin{corollary} Let $z:[0,1]\to [0,1]$ be a non--trivial solution of $z=(1-z)\chi[z]$. Then for all $x\in [0,1]$, we have $\chi[z](x)\in M_\chi:=[k_1-1, K-1]\;$.
\end{corollary}
Also, we conclude the following
\begin{corollary}
If $K=\max_{x} \int_0^1 k(\abs{x-y})\, dy \le 1$, then no non--trivial solution of eqn. \eqref{E:SIS} can exist.
\end{corollary}
To prove the existence of non--trivial solutions we will apply a fix point argument. As we know from previous results, possible non--trivial solutions have to satisfy the bounds of Lem. \ref{L:2}. Therefore we define
\begin{align}
    M_z:= \left\{ z\in \mathcal C^0([0,1];\R):\ \frac{k_1-1}{k_1}\le z(x) \le \frac{K-1}{K}\right\}
\end{align}
as the subset of continuous functions satisfying the needed bounds. It is immediate to see that for $K>1$, the set $M_z$ is non--empty. We equip $M$ with the usual sup-norm
\begin{align}
    \norm{z}_\infty := \max_{x\in [0,1]} \abs{z(x)}\;.
\end{align}
\begin{lemma}
The operator $\Phi[z]:= {\chi[z]}/({1+\chi[z]})$ is a self--mapping on $M_z$, i.e.~$\Phi:M_z\to M_z$.
\end{lemma}
\begin{proof}
Let $z\in M$. Then $k_1-1\le \chi[z]\le K-1$ and, due to the monotonicity of $x\mapsto {x}/({1+x})$, we get
\begin{align}
    \frac{k_1-1}{k_1} \le \Phi[z]=\frac{\chi[z]}{1+\chi[z]} \le \frac{K-1}{K}\;,
\end{align}
hence $\Phi[z]\in M_z$.
\end{proof}
\begin{lemma} If $K/k_1^2<1$, then the operator $\Phi:(M,\norm{\cdot}_\infty)\to (M,\norm{\cdot}_\infty)$ is a contraction on $M$ with respect to the sup--norm.
\end{lemma}
\begin{proof}
Let $u,v\in M$, $x\in [0,1]$ and $K/k_1^2<1$. Then
\begin{align}
    \Phi[u](x)-\Phi[v](x) &= \frac{\chi[u](x)}{1+\chi[u](x)} - \frac{\chi[v](x)}{1+\chi[v](x)} 
    \\&= \frac{\chi[u-v](x)}{1+\chi[u](x)+\chi[v](x)+\chi[u](x)\cdot \chi[v](x)}
\intertext{and}
    \norm{\Phi[u]-\Phi[v]}_\infty &\le \frac{\displaystyle\int_0^1 [u-v](y)\cdot k(\abs{x-y})\,dy}{1+2(k_1-1)+(k_1-1)^2} \le \frac{K}{k_1^2}\cdot \norm{u-v}_\infty< \norm{u-v}_\infty\;.
\end{align}
\end{proof}
\begin{theorem}
\label{T:ExBanach}
For $K/k_1^2<1$, there exists a unique non--trivial solution $z^\ast\in M_z$ of the fix point problem~\eqref{E:SIS}.
\end{theorem}
\begin{proof}
The assertion follows from Banach's fix point theorem applied to the operator $\Phi$ on the set $M_z\subset (\mathcal C^0,\norm{\cdot}_\infty)$.
\end{proof}
\begin{remark}
If $k_1<1$, a solution can be the trivial equilibrium due to $0 \in M_z=\mathcal C^0[k_1-1,K-1]$.
\end{remark}
While no uniqueness properties can yet be made for the case $K/k_1^2\geq 1$, existence can be guaranteed by
\begin{lemma}
If $K>1$ and $k_1>1$, the iteration $\Phi$ has at least one nontrivial solution $z^*(x)\in M_z$.
\end{lemma}
\begin{proof}
$M$ is a convex closed subset of the Banach space $\mathcal C^0(\R, \norm{\cdot})$. It is also non-empty provided $K>1$. The lemma of Arzelà-Ascoli says that for a compact and metric space $(S,d)$ and $M \subset \mathcal C(S)$ equipped with the supremum norm, it holds that $M$ is compact if it is limited, closed, and equicontinuous. This holds for any $\mathcal C[a,b]$ (for a proof see e.g. Werner\cite{Wer95}) and thus for $M_z=\mathcal C[k_1-1,K-1]$. Since $\Phi$ is a continuous mapping of $M_z$ (which is a compact and convex subset of a metric space) into itself, it has at least one fix point in $M$ according to the fix point theorem of Schauder\cite{Sch30}.
\end{proof}
However, this result is not satisfying, as the requirement on the kernel $k$ being 'flat enough' to satisfy $K< k_1^2$ is quite strict: Consider a standard SIS model without any spatial terms. This would be equivalent to the above model with a Dirac delta kernel at $\abs{x-y}=0$, which obviously does \textit{not} satisfy the above equation. However, numerical findings (cf. Appendix A) suggest that a unique non-trivial equilibrium for model \eqref{E:SIS} even for $K/k_1^2>1$ and also $k_1\geq1$ always exists; the same holds true for $k_1<1$, yet $\norm k \geq1$. 

\end{document}